\documentclass [12pt,a4paper,reqno]{amsart}
 \input amssymb.sty
\usepackage{xypic}

 \input xy
\xyoption{all}

\usepackage{thmtools}
\declaretheoremstyle[notefont=\bfseries,notebraces={}{},%
    headpunct={},postheadspace=1em]{mystyle}
\declaretheorem[style=mystyle,numbered=no,name=Problem]{prob-hand}

\def\dispace{\setlength{\itemsep}{2pt}}

\newcommand{\eq}[1]{Eq. (#1)}

\newcommand{\etype}[1]{\renewcommand{\labelenumi}{(#1{enumi})}}
\def\eroman{\etype{\roman}}

\def\pSkip{\vskip 1.5mm \noindent}
\newcommand{\ds}[1]{\ {#1} \ }
\newcommand{\dss}[1]{\quad {#1} \quad }

\def\sm{\setminus}
\def\00{ \{ 0 \}}

\newcommand{\cass}[2]{{#1}{#2}}

\def\|{\ds |}

\def\tL{\mathcal L}
\def\veps{\varepsilon}
\def\vrp{\varphi}
\def\zt{\zeta}

\def\R{\mathbb R}

\def\rk{\operatorname{rk}}
\def\an{\operatorname{an}}

\def\CS{\operatorname{CS}}
\def\X1{X_1}
\def\Y1{Y_1}

\def\map{\phi}

\def\elm{\xi}

\def\nucong{\cong_\nu}
\def\nug{>_\nu}
\def\nul{<_\nu}
\def\nuge{\geq_\nu}
\def\nule{\leq_\nu}

\def\tlb{\tilde b}
\def\tlq{\tilde q}

\def\tT{\mathcal T}
\def\tG{\mathcal G}

\def\N{\mathbb N}

\def\QL{\operatorname{QL}}

\newcommand{\Char}{\operatorname{char}}

\def\an{\operatorname{an}}
\def\CS{\operatorname{CS}}
\def\X1{X_1}
\def\Y1{Y_1}

\def\eps{\varepsilon}

\def\al{\alpha}
\def\bt{\beta}
\def\gm{\gamma}
\def\sig{\sigma}

\newtheorem{thm}{Theorem} [section]
\newtheorem*{thm*}{Theorem}
\newtheorem{cor}[thm]{Corollary}

\newtheorem{lemma}[thm]{Lemma}
\newtheorem{prop}[thm]{Proposition}

\newtheorem{convention}[thm]{Convention}
\newtheorem*{claim*} {Claim}
\newtheorem*{theorem4.5'} {Theorem 4.5$'$}
\newtheorem{acknowledgment*}[thm] {Acknowledgment}

\newtheorem{examp}[thm]{Example}

\newtheorem{examples}[thm]{Examples}
\newtheorem*{exampleA*}{Example A}
\newtheorem*{exampleB*}{Example B}
 \newtheorem{remark}[thm]{Remark}

 \newtheorem*{remark*}{Remark}
 \newtheorem{defn}[thm]{Definition}

\newtheorem{notation}[thm]{Notation}
\newtheorem{notations}[thm]{Notations}
\newtheorem{problem}[thm]{Problem}
\newtheorem*{notation*} {Notation}
\newtheorem*{comment*} {Comment}
\newtheorem{comment}[thm] {Comment}

\newcommand{\thmref}[1]{Theorem~\ref{#1}}
\newcommand{\propref}[1]{Proposition~\ref{#1}}

\newcommand{\lemref}[1]{Lemma~\ref{#1}}
\newcommand{\corref}[1]{Corollary~\ref{#1}}

\newcommand{\Qthmref}[1]{\cite[Theorem~{#1}]{QF1}}

\newcommand{\Qpropref}[1]{\cite[Proposition~{#1}]{QF1}}
\newcommand{\Qsecref}[1]{\cite[~\S{#1}]{QF1}}

\newcommand{\Qdefref}[1]{\cite[Definition~{#1}]{QF1}}

 \renewcommand{\sectionmark}[1]{}

\newcommand{\bfem}[1]{\textbf{#1}}

\newcommand{\lm}{\lambda}

 \newcommand{\supp} {\operatorname{supp}}

\begin{document}

\title[Supertropical  Quadratic Forms II]
{Supertropical  Quadratic Forms II}
\author[Z. Izhakian]{Zur Izhakian}
\address{Institute  of Mathematics,
 University of Aberdeen, AB24 3UE,
Aberdeen,  UK.}
    \email{zzur@abdn.ac.uk; zzur@math.biu.ac.il}

\author[M. Knebusch]{Manfred Knebusch}
\address{Department of Mathematics,
NWF-I Mathematik, Universit\"at Regensburg 93040 Regensburg,
Germany} \email{manfred.knebusch@mathematik.uni-regensburg.de}
\author[L. Rowen]{Louis Rowen}
 \address{Department of Mathematics,
 Bar-Ilan University,  Ramat-Gan 52900, Israel}
 \email{rowen@math.biu.ac.il}






\date{\today}






\begin{abstract} This article is a sequel of \cite{QF1}, where we introduced quadratic forms on a module~ $V$ over a supertropical semiring $R$ and analysed the set of bilinear companions of   a quadratic form $q: V \to R$ in case that the module~$V$ is free, with fairly complete results if~ $R$ is a supersemifield. Given such a companion $b$ we now classify the pairs of vectors in $V$ in terms of $(q,b).$ This amounts to a kind of tropical trigonometry with a sharp distinction between the cases that a sort of Cauchy-Schwarz
inequality holds or fails. We apply this to study the supertropicalizations (cf.  \cite{QF1})  of a quadratic form on a free module~$X$ over a field in the simplest cases of interest where $\rk(X) = 2$.

In the last part of the paper we start exploiting  the fact that the free module~$V$ as  above has a unique base up to permutations and multiplication by units of~$R$, and moreover~$V$ carries a so called minimal (partial) ordering. Under mild restriction on~$R$ we determine all $q$-minimal vectors in $V$, i.e., the vectors $x \in V$ for which  $q(x') < q(x)$ whenever $x' < x.$
\end{abstract}

\maketitle

{\small \tableofcontents}

\numberwithin{equation}{section}

\section*{Introduction}

Let $R$ be a semiring, here always assumed to be commutative and
with 1. A \bfem{quadratic form} on an $R$-module $V$ is a function
$q: V\to R$ with
\begin{equation}\label{eq:0.1}
q(ax)=a^2q(x)
\end{equation}
for any $a\in R$, $x\in V,$ such that there exists a symmetric
bilinear form $b:V\times V\to R$ (not necessarily uniquely
determined by $q$) with
\begin{equation}\label{eq:0.2}
q(x+y)= q(x)+q(y)+b(x,y)
\end{equation}
for any $x,y\in V.$ Every such bilinear form $b$ is called a
\bfem{companion} of $q,$ and the pair $(q,b)$ is called a
\bfem{quadratic pair} on $V.$

The present paper is devoted to a study of quadratic forms and
pairs on $R$-modules with~$R$ a ``supertropical'' semiring,
often more specifically a ``supersemifield''. It is a sequel to the paper
\cite{QF1} by the same authors.

We recall (\Qdefref{0.3} and \cite[\S3]{IKR1}), that a semiring $R$ is
called \bfem{supertropical} if $e:=1_R+1_R$ is an idempotent
(i.e., $2\times 1=4\times 1)$, and the following axioms hold for
all~ $x,y\in R:$
\begin{align}
\text{If}\ ex\ne ey, & \ \text{ then} \ x+y\in\{x,y\}, \label{eq:0.3} \\
\text{If}\ ex= ey,& \ \text{ then} \ x+y=ey. \label{eq:0.4}
\end{align}
Then the ideal $eR$ of $R$ is a semiring with unit element $e,$
which is \bfem{bipotent}, i.e., for any $u,v\in eR$ the sum $u+v$
is either $u$ or $v.$ It follows that $eR$ carries a total
ordering, compatible with addition and multiplication, which is
given by
\begin{equation}\label{eq:0.5}
u\le v \ds \Leftrightarrow u+v=v.
\end{equation}
The addition in a supertropical semiring is determined by the map
$x\mapsto ex$ and the total ordering on $eR$ as follows: If
$x,y\in R,$ then
\begin{equation}\label{eq:0.6}
x+y=\begin{cases} y&\ \text{if}\ ex<ey,\\
x&\ \text{if}\ ex>ey,\\
ey&\ \text{if}\ ex=ey.
\end{cases}
\end{equation}
In particular (taking  $y=0$ in \eqref{eq:0.6} or in \eqref{eq:0.4})
\begin{equation}\label{eq:0.7}
ex=0 \ds \Rightarrow x=0.
\end{equation}

For the convenience of the reader, we give more terminology. In a
supertropical semiring~$R,$ the elements of the set
$\tT(R):=R\sm(eR)$ are called \bfem{tangible}, while those of
the set $\tG(R):=(eR)\sm\{0\}$ are called
\bfem{ghost} elements. The zero of $R$ is regarded both as
tangible and ghost. The semiring $R$ itself is called \bfem{tangible} if $R$ is
generated by $\tT(R)$ as a semiring. Clearly, this happens iff
$e\tT(R)=\tG(R).$ If $\tT(R)\ne\emptyset,$ then the set
$$R':=\tT(R)\cup e\tT(R)\cup\{0\}$$
is the largest subsemiring of $R$ which is tangible supertropical.
\{We have discarded the ``superfluous'' ghost elements.\}

In the paper \cite{QF1}, the main thrust is the study of the set of
all companions of a given quadratic form $q$ on a free module $V$
over a supertropical semiring $R.$ After fixing a base
$(\varepsilon_i\|i\in I)$ of $V$, this set can be described by use
of a ``companion matrix'' ($C_{i,j}(q))$, cf.~ \Qsecref{6}.  For
$R$ a tangible semifield, complete results can be found in
\Qsecref{7}. Explicitly, these hard results are needed in the
present paper only in the proof of the initial key Theorem
\ref{thm:II.5.5}, which for a first reading may be taken on
faith.

The quadratic form $q$ is called \bfem{rigid}, if $q$ has only one
companion. This happens iff $q(\veps_i)=0$ for all vectors
$\veps_i$ of the fixed base $(\veps_i\|i\in I),$ cf.
\Qthmref{3.5}. $q$ is called \bfem{quasilinear} if the
bilinear form $b=0$ is a companion of $q,$ i.e.,
$q(x+y)=q(x)+q(y)$ for all $x,y\in V.$ These are the ``diagonal''
forms on $V,$
\begin{equation}\label{eq:0.8}
q\bigg(\sum_ix_i\veps_i\bigg)=\sum_i q(\veps_i)x_i^2,
\end{equation}
due to the fact that $(\lm+\mu)^2=\lm^2+\mu^2$ for all
$\lm,\mu\in R,$ cf. \Qpropref{0.5}.

Any quadratic form $q$ on a free $R$-module can be written as a
sum
\begin{equation}\label{eq:0.9}
q=q_{QL}+\rho,
\end{equation}
where $q_{QL}$ is a quasilinear (and uniquely determined by $q$)
and $\rho$ is rigid (but not unique), cf. \Qsecref{4}. We call
$q_{\QL}$ the \bfem{quasilinear part} of $q$ and $\rho$ a
\bfem{rigid complement} of $q_{\QL}$ in $q.$

\medskip
The present paper is divided as follows. The first three sections
are devoted to a study of pairs of non-zero vectors $(x,y)$ in  an
$R$-module $V$ equipped with a quadratic pair $(q, b),$ mostly for
$R$ a tangible semifield. Sometimes we only assume that $eR$ is a
(bipotent) semifield. We face an all important dichotomy. Either
$(x,y)$ is \textit{excessive} (cf. Definition \ref{defn:II.5.6}
below) or the restriction $q|Rx+Ry$ of $q$ is quasilinear. In the
latter case, we also say that the pair $(x,y)$ \bfem{quasilinear}
(with respect to $q)$.

An intriguing point here is that this dichotomy does not depend on
the choice of the companion $b$ of $q,$ although $b$ is used in
the definition of excessiveness (cf. \corref{cor:II.5.7}).

In Section \ref{sec:II.6}, we delve into a kind of ``\textit{tropical
trigonometry}''. If $x$ and $y$ are anisotropic, i.e., $q(x)\ne0,$
$q(y)\ne0,$ we define a CS-ratio\footnote{``CS'' is an acronym of
``Cauchy-Schwarz''.}
\begin{equation}\label{eq:0.10}
\CS(x,y):=\frac{eb(x,y)^2}{eq(x)q(y)}\in eR,
\end{equation}
which makes sense since $eR$ is a semifield.
When the set $eR$ is densely ordered, then $(x,y)$ is excessive
iff $\CS(x,y)>e.$ When $eR$ is discrete, the pair $(x,y)$ is
excessive if $\CS(x,y)>c_0,$ with $c_0$ the smallest element of
$eR$ bigger than $e.$ But if $\CS(x,y)=c_0,$ the pair $(x,y)$ is
excessive if $q(x)$ or $q(y)$ is tangible, while $(x,y)$ is
quasilinear if both $q(x)$ and $q(y)$ are ghost (cf. Theorems
\ref{thm:II.5.5} and \ref{thm:II.5.12}). It seems to us that
this still somewhat mysterious fact bears relevance for problems
of an arithmetical nature in quadratic form theory, even over
fields.

For any anisotropic vector $w$, the function $x\mapsto
\CS(x,w)$ is subadditive, cf. \thmref{thm:II.6.10}. This fact
has turned out to be of central importance in a (still incomplete)
sequel \cite{QF3} of the present paper.

In \S\ref{sec:II.7}, we compile tables of the function
$(\lm,\mu)\mapsto q(\lm x+\mu y)$ on $(R\sm\{0\})^2$
for given $x,y\in V\sm\{0\},$ and then study in detail the
CS-ratios $\CS(x',y')$ of pairs of vectors $(x',y')$ in $Rx+Ry.$
This completes our account of tropical trigonometry in the present
paper. First applications show up in the later sections, but a
more adequate language of ``rays'',\footnote{The ray of a vector
$x\in V\sm\{0\}$ is the set of all $y\in V$ with $\lm
x=\mu y$ for some $\lm,\mu\in R\sm\{0\}.$} to use this
trigonometry conveniently, has to wait for the paper \cite{QF3}
due to lack of space here.

Sections \S\ref{sec:II.8}--\S\ref{sec:7} of the paper are based on
the following two facts for $R$-modules, valid over any
supertropical  semiring $R:$
\begin{enumerate} \dispace
\item[1)] \textit{The Unique Base Theorem}, cf. \Qthmref{0.9}:
Given a base $(\veps_i\|i\in I)$ of a free $R$-module $V,$ we
obtain any other base of $V$ by permuting the $\veps_i$ and
multiplying them by units of $R.$

\item[2)] \textit{Existence of minimal orderings}, cf.
\S\ref{sec:I.6} below. Every $R$-module $V$ carries a partial
ordering, called the \bfem{minimal ordering} on $V,$ which is
defined as follows:
$$x\le y \ds\Leftrightarrow \exists z\in V: x+z=y.$$
In particular, $R$ itself has a minimal ordering. The minimal
ordering on $V$ is compatible with addition and scalar
multiplication. Basics about the minimal ordering on $R$ and then
on a free $R$-module are provided in \S\ref{sec:I.6}.
\end{enumerate}

The Unique Base Theorem is the source of our motivation for
introducing supertropicalizations of a quadratic form $q:V\to R$
on a free module $V$ over a ring $R$ by a so-called
\textit{supervaluation} $\varphi:R\to U$ with values in a
supertropical semiring $U$ in \Qsecref{9}. Given a base
$\tL=(\veps_i\|i\in I)$ of $V$, we obtained a
quadratic form $\tlq:U^{(I)}\to U$ on the standard free
$U$-module $U^{(I)}$ by this process [loc. cit.], which in some
sense measures $\tL$ in terms of $q$ and $\varphi.$ In
\S\ref{sec:II.8} of the present paper, we study how $\tilde q$ varies
with a change of the base $\tL$ in the simplest cases of
interest, where $I=\{1,2\}.$

Given a quadratic form $q:V\to R$ on a module $V$ over a
supertropical semiring $R,$ we call a vector $x\in V$
$q$-\bfem{minimal}, if $q(x')<q(x)$ for every vector $x'<x$ (with
respect to the minimal ordering of $V$ and $R)$.\footnote{As
usual, $x'<x$ means $x'\le x$ and $x'\ne x.$}

In the last sections \S\ref{sec:II.9} and \S\ref{sec:7}, we obtain a
detailed description of  all minimal vectors and certain relations between
them in the case that $V$ is free and $R$ is tangible
supertropical with $\tG(R)$ a cancellative monoid under
multiplication (in particular, if $R$ is a tangible
supersemifield).

Every $q$-minimal vector $x \in V$ is trapped in a smallest submodule $V_J = \sum_{i \in J} R v_i$ of $V$ with $|J| \leq 4$, and thus it suffices to study $q$-minimal vectors in a given free module of rank at most $4$. In \S\ref{sec:II.9} we easily find all $q$-minimal vectors for $|J| \leq 2$ (vectors of ``small support''). Then in
\S\ref{sec:7}  we prove that for $|J| = 3$ or $|J| = 4$ a  $q$-minimal vector $x$ is the maximum $y \vee z$ of a pair of  $q$-minimals $y$ and $z$ of small support which is  uniquely determined by $x$, except in one case, where $y$ and $z$ can be freely chosen in a triplet $y_1, y_2, y_3$ of  $q$-minimals of small support,  uniquely determined by $x$. Conversely, we find out which maxima $y \vee z$ of $q$-minimals $y,z$ with small support are again $q$-minimal.

The arguments in \S\ref{sec:II.9} and \S\ref{sec:7} may look massy due to the many case distinctions needed, but the give a good illustration of the, as we feel, beautiful combinatorics at hands in any supertropical quadratic space.

\begin{notation}\label{notation:0.1}
Let $\mathbb N=\{1,2,3,\dots\}$, $\mathbb N_0=\mathbb N\cup\{0\}.$
If $R$ is a semiring, then $R^*$ denotes the group of units of
$R.$

If $R$ is a supertropical semiring, then
\begin{enumerate} \dispace
\item[$\bullet$] $\tT(R):=R\sm eR=$ set of tangible
elements $\ne 0.$

 \item[$\bullet$] $\tG(R):=eR\sm \{0\}=$
set of ghost elements $\ne 0.$

\item[$\bullet$] $\nu_R$ denotes the ghost map $\to eR,$ $a\mapsto
ea.$
\end{enumerate}
When there is no ambiguity, we write $\tT,$ $\tG,$
$\nu$ instead of $\tT(R),$ $\tG(R),$
$\nu_R.$\newline For $a\in R$ we also write $ea=\nu(a)=a^\nu.$ $a \leq_\nu b$  means that $ea  \leq eb$, $a \nucong b$ (``$\nu$-equivalent'') means that $ea  = eb$, while  $a <_\nu b$  means that $ea  < eb$.
\end{notation}

\section{Pairs of vectors in a supertropical quadratic space}\label{sec:II.5}

\begin{defn}\label{defn:II.5.1}
\quad{}
\begin{enumerate}
  \item[ a)] A \bfem{quadratic module} over a semiring $R$ is a pair
 $(V,q)$ consisting of an $R$-module~$V$ and a (functional)
 quadratic form $q$ on $V.$ Later we often will  write  a
 single letter $V$ instead of $(V,q).$ \pSkip

\item[ b)] A \bfem{supertropical quadratic space} is a quadratic module
 over a tangible supersemifield.
 \end{enumerate}

 \end{defn}

We intend to study pairs of vectors in a supertropical quadratic
space. Preparing for this we slightly extend the notion of ``partial'' rigidity
developed in \Qsecref{3} (cf. \Qdefref{3.1}). This makes
sense over any supertropical semiring $R.$

\begin{defn}\label{defn:II.5.2}
Let $(V,q)$ be a quadratic module over a supertropical semiring
$R.$ We say that $q$ is $\nu$-\bfem{rigid at a point} $(x,y)$ of
$V\times V$ if
\begin{equation}\label{eq:II.5.1}
eb_1(x,y)=eb_2(x,y)\end{equation} for any two companions $b_1,b_2$
of $q,$ and we say that $q$ is $\nu$-\bfem{rigid on a set}
$T\subset V\times V$ or on a set $S\subset V,$ if this happens for
all $(x,y)$ in $T$ or in $S\times S$, respectively.
\end{defn}


If the $R$-module $V$ is free with base $(\veps_i \|i\in I)$,
then $\nu$-rigidity of $q$ at $(\veps_i,\veps_j)$
means that all $\bt\in C_{i,j}(q)$ have the same ghost value,
i.e., the set $e\cdot C_{i,j}(q)$ is a singleton. We have seen the
phenomenon of $\nu$-rigidity (beyond rigidity) already in   equation~$(6.5)$ of
\Qthmref{6.9}.
\pSkip

Assume as before that $R$ is a supertropical semiring $R,$ and
that $(V,q)$ is a quadratic module over $R.$ Given a pair of
vectors $(x,y)\in V\times V$, we have a unique  $R$-linear map
\begin{equation}\label{eq:II.5.2}
\chi:=\chi_{x,y}:R\veps_1+R\veps_2\to V\end{equation}
from the free $R$-module $R\veps_1+R\veps_2$ with base
$\veps_1,\veps_2$ to $V$ such that
$\chi(\veps_1)=x,$ $\chi(\veps_2)=y.$ This map $\chi$
composes with $q:V\to R$ to a quadratic form
\begin{equation}\label{eq:II.5.3} \tlq :=q\circ\chi:R\veps_1+R\veps_2\to R.\end{equation}

\begin{prop}\label{prop:II.5.3}
\quad{}

\begin{enumerate}\item[i)] If $b:V\times V\to R$ is a companion of
$q,$ then the symmetric bilinear form $\tlb$ on
$R\veps_1+R\veps_2$ defined by \begin{equation}\label{eq:II.5.4}
\tlb(v_1,v_2):=b(\chi(v_1),\chi(v_2)) \qquad (v_1,v_2\in V) \end{equation} is a
companion of $\tlq .$
\pSkip

\item[ii)] If $\tlq $ is rigid at
$(\veps_1,\veps_2),$ then $q$ is rigid at $(x,y).$

\pSkip
\item[iii)] If $\tlq $ is $\nu$-rigid at
$(\veps_1,\veps_2),$ then $q$ is $\nu$-rigid at $(x,y).$
\end{enumerate}
\end{prop}

\begin{proof} Claim i) follows directly from the definition of a
companion in \Qsecref{1} (\Qdefref{1.14}).

Claims ii) and iii) are immediate consequences of i).\end{proof}

Concerning quasilinearity, we have a stronger statement.

\begin{prop}\label{prop:II.5.4} Given $(x,y)\in V\times V$, the
following are equivalent.
\begin{enumerate} \eroman \item $q$ is quasilinear on $Rx \times Ry.$
\pSkip

\item$q$ is quasilinear on $Rx+Ry.$
\pSkip

\item
$\tlq $ is quasilinear on   $ R \veps_1 \times R \veps_2.$
\pSkip

\item $\tlq $ is quasilinear.
\end{enumerate}\end{prop}

\begin{proof}
Condition (iii) means that $0\in C_{1,2}(\tlq ).$ Since  $0\in
C_{i,i}(\tlq )$ holds for $i=1,2,$ it is clear from \Qsecref{5} that (iii)
$\Leftrightarrow$ (iv). \pSkip

(ii) means that $q$ is additive on $Rx + Ry,$ while (iv) means
that $\tlq $ is additive. Thus the equivalence (ii)
$\Leftrightarrow$ (iv) follows from the additivity and
surjectivity of $\chi$ as a map from
$R\veps_1+R\veps_2$ to $Rx+Ry.$ \pSkip

(i) means that $q(\lm x+ \mu y)=q(\lm x)+q(\mu y)$, and (iii) means that $$\tlq (\lm \veps_1+ \mu \veps_2)=\tlq (\lm \veps_1)+\tlq (\mu \veps_2)$$ for all $\lm, \mu \in R$ (cf. \Qdefref{2.3}). Thus clearly (i)
$\Leftrightarrow$ (iii).

We conclude that all four conditions (i) -- (iv) are
equivalent.\end{proof}

We are ready for a key theorem of the paper, emanating from \Qsecref{7}.

\begin{thm}\label{thm:II.5.5}
Assume that $R$ is a nontrivial tangible supersemifield and $(q,b)$ is a quadratic
pair on an $R$-module $V.$ Let $(x,y)$ be a pair of vectors in
$V.$ We adhere to \cite[Terminology~7.7]{QF1}.
\begin{enumerate}\item[a)] Assume that $R$ is dense. Then $q$ is
quasilinear on $Rx+Ry$ iff
\begin{equation}\label{eq:II.5.5}
b(x,y)^2\le_\nu q(x)q(y).\end{equation}
Otherwise $q$ is rigid at $(x,y)$.
\pSkip

\item[b)] Assume that $R$ is discrete with $\pi$ a prime element of
$R$.

Now $q$ is quasilinear on $Rx+Ry$ if either
\begin{equation}\label{eq:II.5.6}
b(x,y)^2<_\nu \pi^{-1}q(x)q(y)\end{equation} or both values
$q(x),q(y)$ are ghost and
\begin{equation}\label{eq:II.5.7}
b(x,y)^2\cong_\nu \pi^{-1}q(x)q(y).\end{equation}
 Otherwise $q$ is $\nu$-rigid at $(x,y). $ If
 \begin{equation}\label{eq:II.5.8}
 b(x,y)^2>_\nu \pi^{-1}q(x)q(y)\end{equation}
 then $q$ is rigid at $(x,y).$\end{enumerate}
\end{thm}

\begin{proof}
By Propositions \ref{prop:II.5.3} and \ref{prop:II.5.4} above it
suffices to prove these claims in the special case that $V$ is
free with base $\veps_1,\veps_2$ and
$x=\veps_1,$ $y=\veps_2.$ Now the results can be read
off from \Qpropref{7.9} and \cite[Theorems 7.11 and 7.12]{QF1}.
\end{proof}

In order to obtain a better grasp on the contents of this theorem,
we introduce more terminology. \textit{As before $R$ is a nontrivial tangible
supersemifield}.

\begin{defn}\label{defn:II.5.6}
Assume that $(q,b)$ is a quadratic pair on an $R$-module $V.$ We
say that a pair of vectors $(x,y)\in V\times V$ is
\bfem{excessive} (w.r.t. $(q,b)$), if the following holds:
\begin{enumerate}
\item[a)] If $R$ is dense, then
$$b(x,y)^2>_\nu q(x)q(y).$$
\pSkip

\item[b)] If $R$ is discrete, then either
$$b(x,y)^2>_\nu \pi^{-1}q(x)q(y),$$
or
$$b(x,y)^2\cong_\nu \pi^{-1}q(x)q(y)$$
and $q(x)\in \tT$ or $q(y)\in \tT.$\end{enumerate}\end{defn}

\thmref{thm:II.5.5}, up to the rigidity statements there, can be
reformulated as follows.

\begin{cor}\label{cor:II.5.7}
A pair $(x,y)\in V\times V$ is excessive with respect to $(q,b)$
iff $q$ is \bfem{not} quasilinear on $Rx+Ry.$\end{cor}

An intriguing point here is that the property ``excessive''
depends only on $x,y,q.$ The choice of the companion $b$ has no
influence, but, of course, is relevant for deciding by computation
whether $(x,y)$ is excessive or not.

We state an easy consequence of \corref{cor:II.5.7}.

\begin{prop}\label{prop:II.5.8}
Let $(x,y)$ and $(x',y')$ be pairs of vectors in a quadratic space
$(V,q)$ over a tangible supersemifield. Assume that $(x',y')$ is
excessive and $Rx'+Ry'\subset Rx+Ry.$ Then $(x,y)$ is excessive.
\end{prop}

\begin{proof} Otherwise $q$ would be quasilinear on $Rx+Ry.$ But
this implies that $q$ is quasilinear on $Rx'+Ry',$ a
contradiction.\end{proof}

We now relax the assumption that $R$ is a tangible supersemifield and demonstrate that several results  obtained so far in the section  remain valid in greater generality.

\begin{convention}\label{conv:II.5.9}
We only assume that $R$ is a supertropical semiring and $eR$ is a semifield, i.e., every element of $\tG = eR \sm \00 $  is invertible in $eR;$ hence $\tG$ is a totaly ordered group.  Moreover we assume that $eR$  is ``nontrivial'', i.e., $\tG \neq \{ e\}$. We do not assume anything about $\tT := R \sm eR.$ ($\tT$ may even be empty.) We call $\tG$ \textbf{discrete}, if $\tG$ contains a smallest element $c > e$, which we denote by $c_0.$ (If $R$ is a tangible supersemifield then $c_0 = e \pi ^{-1}$ in
the setting \cite[Terminology~7.7]{QF1}.)
Otherwise we call $\tG$ \textbf{dense}.
\end{convention}

Assume in the following that $(q,b)$ is a quadratic pair on the $R$-module $V.$
For the sake of brevity   we call a pair $x,y$ of vectors in $V \sm \{ 0 \}$  \textbf{quasilinear} if $q$ is quasilinear on
$Rx \times Ry$, equivalently if the restriction  $q | Rx \times Ry$ of $q$ is quasilinear.

\begin{defn}\label{def:II.5.10}
We say that a pair of vectors $x,y$ in $V \sm \00 $ is \textbf{CS} (acronym for ``Cauchy-Schwarz''), if
\begin{equation}\label{eq:II.5.9}
b(x,y)^2 \nul q(x)q(y).\end{equation}
We call $(x,y)$ \textbf{weakly CS}, if
\begin{equation}\label{eq:II.5.10}
b(x,y)^2 \nule q(x)q(y)\end{equation}
(a condition already appearing in \eqref{eq:II.5.5}), and we call $(x,y)$
 \textbf{almost CS}, if
\begin{equation}\label{eq:II.5.11}
b(x,y)^2 \nule c q(x)q(y)\end{equation}
for all $c > e$ in $\tG.$ \footnote{In \cite[\S5]{IKR-LinAlg2} the terms ``CS'' and ``weakly CS'' have been used in a similar way for pairs of vectors with respect to a (not necessarily symmetric) bilinear form. }
\end{defn}

\begin{remark}\label{rem:II.5.11}
Assume that $(x,y)$ is almost CS. If $\tG$ is dense, then $(x,y)$ is weakly CS, whereas if $\tG$ is discrete, either $(x,y) $ is weakly CS, or $b(x,y)^2 \nucong c_0 q(x)q(y)$.
\end{remark}
We save a relevant part of Theorem \ref{thm:II.5.5} in the present more general situation.

\begin{thm}\label{thm:II.5.12}
If either $(x,y)$ is weakly CS, or  $(x,y)$ is almost  CS and both $q(x) $ and $q(y)$ are ghost, then $(x,y)$  is quasilinear.
\end{thm}

\begin{proof}
If $(x,y)$ satisfies the assumptions of the theorem then so does $(\lm x, \mu y)$ for all $\lm, \mu \in R \sm \00.$
Thus in  view of  Proposition \ref{prop:II.5.4} it  suffices to prove that
 \begin{equation}\renewcommand{\theequation}{$*$}\addtocounter{equation}{-1}\label{eq:str.1}
q(x+y) = q(x) + q(y).
\end{equation}
In general we have
 \begin{equation}\renewcommand{\theequation}{$**$}\addtocounter{equation}{-1}\label{eq:str.1}
q(x+y) = q(x) + q(y) + b(x,y).
\end{equation}
If $b(x,y)^2 \nul q(x) q(y)$, then either
$b(x,y) \nul q(x) $ or $b(x,y) \nul  q(y)$, and the summand $b(x,y)$ in $(**)$ can be omitted, giving $(*).$
\pSkip

Assume now that $b(x,y)^2 \nucong q(x) q(y)$. If $q(x), q(y)$ are not $\nu$-equivalent, say $q(x) \nul q(y)$, then $b(x,y)^2 \nul q(y)^2$, hence  $b(x,y) \nul q(y)$, and again the term $b(x,y)$ can be omitted in~$(**)$.  If $q(x) \nucong q(y)$ then we have  $b(x,y)^2 \nucong q(x)^2$, hence
 $b(x,y) \nucong q(x) \nucong q(y),$ and the right hand side of~$(**)$ equals $eq(x) = q(x) + q(y).$ Thus $(*) $ holds again. \pSkip

There remains the case  that $\tG$ is discrete and $b(x,y)^2 \nucong c_0 q(x) q(y).$ Now $q(x)q(y)$ is not  a $\nu$-square. We may assume   that $q(x) \nul q(y)$. Now $c_0 q(x) \nule q(y)$. Hence   $b(x,y)^2 \nule q(y)^2$, and  hence  $b(x,y) \nule  q(y)$. If $b(x,y) \nul  q(y)$ we obtain $(*)$ from $(**)$ as before. Otherwise  $b(x,y) \nucong  q(y)$, and hence $q(x+y)= eq(y)= q(x) + eq(y).$ Thus, if $q(y) \in eR,$ then $q(x+ y)  = q(x) + q(y).$
\end{proof}

\begin{remark}\label{rem:II.5.13} The bad case is that $R$ is discrete, with $c_0 q(x) \nucong q(y) \nucong b(x,y),$ perhaps after interchanging~$x$ and $y$, and $q(y)$ is tangible. Then $q(x) + q(y) = q(y),$ while $q(x+y) = q(y)+ b(x,y) = eq(y).$
\end{remark}

\begin{remark}\label{rem:II.5.14}
Let P be any  of the properties in Definition \ref{def:II.5.10} $(\CS, \dots)$ or -- if $R$ is a tangible supersemifield -- one of the conditions in Theorem \ref{thm:II.5.5}. Assume that $\lm,\mu \in \tT.$ Then it is obvious that a pair $(x,y) \in V \times V$ has property P iff $(\lm x, \mu y)$ has property P. Except for the properties discussed  in Theorem \ref{thm:II.5.5}.b involving \eqref{eq:II.5.7},  this even remains true if $\lm, \mu \in R \sm \00 .$
\end{remark}

\section{CS-ratios: Definition and subadditivity }\label{sec:II.6}

If $R$ is any semiring and $q : V \to R$ is  a quadratic form on an $R$-module $V$, we call a vector  $x \in V \sm \{ 0 \}$ \textbf{isotropic} if $q(x) = 0$ and \textbf{anisotropic} if $q(x) \neq 0.$  The zero vector in~$V$ is regarded both as isotropic and  anisotropic. If the semiring $R$ is supertropical, it follows directly from the definition of a quadratic from (cf. \cite[\eq{0.1} and \eq{0.2}]{QF1}) that the set of anisotropic vectors
\begin{equation}\label{eq:II.6.1}
V_{\an}:=\{x\in V \ds |q(x)\ne0\} \cup \{ 0 \},\end{equation}
is an $R$-submodule of $V$, and moreover
\begin{equation}\label{eq:II.6.2}
V+V_{\an} =V_{\an}.\end{equation}

We now always assume in this section that $R$ is supertropical,  that $eR$ \bfem{is a nontrivial bipotent
semifield} (cf. Convention \ref{conv:II.5.9}),  and that $(q,b)$ is \textbf{a
fixed quadratic pair} on $V.$ We develop the concept of
``\textit{CS-ratios}'' for pairs of vectors  in
$V_{\an}.$ To a large extent this may be viewed as a kind of
``trigonometry" in supertropical quadratic spaces. \pSkip

We start with a definition where the quadratic pair is not yet
needed.

\begin{defn}\label{defn:II.6.1}
Given $\lm\in R$ and $\mu\in R\sm\{0\}$ the
$\nu$-\bfem{ratio} $\big[\frac{\lm}{\mu}\big]_\nu$ is the
fraction $\frac{e\lm}{e\mu}$ in the semifield $eR=\tG \cup\{0\}.$ Thus for any $\gamma\in R$
\begin{equation}\label{eq:II.6.3}
\bigg[\frac{\lm}{\mu}\bigg]_\nu\cong_\nu\gamma
\dss\Leftrightarrow\gamma\mu\cong_\nu\lm.
\end{equation}
\end{defn}

This slightly funny notation reflects the desire in supertropical
algebra to work as much as possible with tangible elements.
Indeed, if $R$ happens to be a tangible supersemifield (the most
important case for us), we can write all $\nu$-ratios $\ne0$ as
$\big[\frac{\lm}{\mu}\big]_\nu$ with $\lm,\mu\in \tT.$ Then  the $\nu$-ratio $\big[\frac{\lm}{\mu}\big]_\nu$ is characterized by
\begin{equation}\label{eq:II.6.4} \forall \gm \in \tT: \qquad
\bigg[\frac{\lm}{\mu}\bigg]_\nu\cong_\nu\gamma \dss
\Leftrightarrow\lm\cong_\nu\gamma\mu.
\end{equation}

\begin{defn}\label{defn:II.6.2}
Let $x,y\in V_{\an} \sm \{ 0 \}.$ We call
$$\CS(x,y):=\bigg[\frac{b(x,y)^2}{q(x)q(y)}\bigg]_\nu$$
the \bfem{CS-ratio} of the pair of vectors $(x,y)$ (with respect
to $(q,b)).$\end{defn}

\begin{remark}\label{rem:II.6.3}
In case of anisotropic vectors $x,y$, we can  reformulate Definition
\ref{def:II.5.10} as follows: The pair $(x,y)$ is CS iff
$\CS(x,y)< e;$ weakly
 CS iff $\CS(x,y)\le e;$ and almost CS iff $\CS(x,y)<c$ for any
 $c>e$ in $\tG.$\end{remark}

\begin{remark}\label{rem:II.6.4}
Clearly, $\CS(x,y)=\CS(y,x).$ Notice also that
\begin{equation}\label{eq:II.6.6}
\CS(\lm x,\mu y)=\CS(x,y)\end{equation} for any $\lm,\mu\in
R\sm\{0\}.$
\end{remark}

Given vectors $x,y,w\in V_{\an}$, we look for constraints on the
$\CS$-ratio $\CS(x+y,w)$ in terms of $\CS(x,w)$ and $\CS(y,w).$ We
need a lemma from \cite{IzhakianRowen2007SuperTropical}, (in fact a weak
version of it), reproved here for the convenience of the reader.

\begin{lemma}[cf. {\cite[Lemma 3.16.ii]{IzhakianRowen2007SuperTropical}}.]\label{lem:II.6.9} Assume
as before that $eR$ is a  semifield.\\ Let~$a,b,c,d\in R.$
\begin{enumerate}
\item[i)] If $bc\cong_\nu ad,$ then
\begin{equation}\label{eq:II.6.8}
ac+bd=(a+b)(c+d).\end{equation} \item[ii)] If $a\cong_\nu b,$
\bfem{or} $c\cong_\nu d,$ then  still
\begin{equation}\label{eq:II.6.9}
ac+bd\cong_\nu (a+b)(c+d).\end{equation}
\end{enumerate}
 \end{lemma}

\begin{proof}
i): We assume without loss of generality  that $a\ge_\nu
b.$

\begin{enumerate}
 \item[1.] \textit{Case}: $a\cong_\nu b\ne0.$ Now $c\cong_\nu d.$ Both
 sides of \eqref{eq:II.6.8}  equal $eac.$

  \item[2.] \textit{Case}: $a>_\nu b.$ Now $bc\cong_\nu ad$ implies that
  $c>_\nu d$ or $c=d=0.$ If $c=d=0$, both sides of \eqref{eq:II.6.8}
  are zero. Otherwise $ac>_\nu bd,$ and both sides of
  \eqref{eq:II.6.8} equal $ac.$

 \item[3.] \textit{Case}: $a=b=0.$ Both sides of \eqref{eq:II.6.8} are zero.
\end{enumerate}

 ii): This is evident.
\end{proof}

We now are ready for a theorem, which states subadditivity of the function $x \mapsto \CS(x,w)$ from $V_{\an} \sm \{0 \}$ to $\tG$ for a fixed $w$, together with refinements of this fact.

\begin{thm}\label{thm:II.6.10}
Let $x,y,w$  be anisotropic vectors in $V.$

\begin{enumerate}
\item[a)] Then
\begin{equation}\label{eq:II.6.10}
\CS(x+y,w)\ds\le\CS(x,w)+\CS(y,w).\end{equation}

\item[b)] If $q(x+y)$ is \bfem{not} $\nu$-equivalent to
$q(x)+q(y)$ and also $\CS(x,w)+\CS(y,w)\ne0,$ then
\begin{equation}\label{eq:II.6.11}
\CS(x+y,w)\ds <\CS(x,w)+\CS(y,w).\end{equation}

\item[c)] Assume that $q(x+y)\cong_\nu q(x)+q(y),$ and that either \pSkip

\begin{enumerate}
\item[c1)] $q(x)\CS(y,w)=q(y)\CS(x,w)$

or

\item[c2)] $\CS(x,w)=\CS(y,w)$

or

\item[c3)] $q(x)\cong_\nu q(y).$

 \end{enumerate} \pSkip
Then
\begin{equation}\label{eq:II.6.12}
\CS(x+y,w)\ds=\CS(x,w)+\CS(y,w).\end{equation}
\end{enumerate}
\end{thm}

\begin{proof}
 Let $c:=\CS(x,w),$ $d:=\CS(y,w).$ Thus
$$b(x,w)^2\cong_\nu cq(x)q(w),$$
$$b(y,w)^2\cong_\nu dq(y)q(w).$$
 Adding these two relations and using that
$(\lm+\mu)^2=\lm^2+\mu^2$ for $\lm,\mu\in R,$ we obtain
\begin{equation}\label{eq:II.6.13}
b(x+y,w)^2\cong_\nu[cq(x)+dq(y)]q(w).\end{equation} Putting
$a:=q(x),$ $b:=q(y),$ we trivially have
$$ac+bd\le_\nu(a+b)(c+d),$$
and further
$$a+b\le_\nu q(x)+q(y)+b(x,y)=q(x+y).$$
We conclude that
$$b(x+y,w)^2\le_\nu (a+b)(c+d)q(w)\le_\nu(c+d)q(x+y)q(w).$$
This tells us that $\CS(x+y,w)\le c+d,$  which is claim a) of the
theorem. Moreover, if $a+b<_\nu q(x+y)$ and $c+d\ne0,$ then
$$b(x+y,w)^2<_\nu (c+d)q(x+y)q(w),$$
which is claim b) of the theorem.

Henceforth we assume that $q(x+y)\cong_\nu a+b$ and now have to
prove equation \eqref{eq:II.6.12}. By~\eqref{eq:II.6.13} above the equation
means that
$$ac+bd\cong_\nu (a+b)(c+d).$$
We know by \lemref{lem:II.6.9} that this holds if $ad\cong_\nu bc,$
and also if $a\cong_\nu c$ or $b\cong_\nu d.$ \{We only need the
statement \eqref{eq:II.6.9} in the lemma, leaving the more
interesting assertion \eqref{eq:II.6.8} for later use.\} This proves part c) of
the theorem.\end{proof}

\section{A table of $q$-values, and CS-ratios of pairs of vectors}\label{sec:II.7}

Throughout this  section $V$ is a module over a tangible supersemifield $R$, and $(q,b)$ is a quadratic pair on $V$. We fix a pair of vectors $(x,y) \in V \times V$ and use the abbreviations
\begin{equation}\label{eq:II.7.1}
\al_1 := q(x), \qquad \al_2 := q(y), \qquad \al := b(x,y).
\end{equation}
Our first goal is to compile a table of values of the function $R \times R \to R$, $(\lm,\mu) \mapsto q(\lm x + \mu y)$, using the parameters $\al_1, \al_2, \al.$ We then will use the table  (Propositions \ref{prop:II.7.4} and \ref{prop:II.7.7}) for various purposes here and in the sequels of this paper.

For establishing the table we may replace $V$ by the free module $R \eps_1 + R \eps_2$ with base $\eps_1, \eps_2$, the vector pair $(x,y)$ by $(\eps_1, \eps_2)$, and the quadratic pair $(q,b)$ by the quadratic pair $(\tlq, \tlb)$ on $R \eps_1 + R \eps_2$, obtained by composing $(q,b)$ with the bilinear map $$\chi : R \eps_1 + R \eps_2 \ds \to V$$  with $\chi(\eps_1) = x$, $\chi(\eps_2) = y$, as described in \eqref{eq:II.5.2}--\eqref{eq:II.5.4}. Thus we may assume that $V$ is free with base $x,y$ and $$ q = \begin{bmatrix} \al_1& \al\\   &   \al_2\end{bmatrix}, $$
whenever we feel that this is convenient.

We do not  assume this now, but we extend \cite[Convention 7.10]{QF1} for  the parameters $\al_1, \al_2, \al$ to the present situation in case that $\al_1 \neq 0$  and $\al_2 \neq 0$. Thus we have an element $\elm \in \tT^{1/2}$ with $\al_1 \elm \cong_\nu \al_2, $ and $\elm \in \tT$ if $\al_1 \al_2$ is a $\nu$-square. In the case that $R$ is discrete and $\al_1 \al_2$ is not a $\nu$-square, we furthermore  have elements $\sig, \tau$ in $\tT$, such that $e \tau < e \sig$ and $e \tau, e\sig$ are the elements of $\tG$ nearest to $e \elm$ in the totally ordered set $\tG^{1/2},$ i.e. $\tau <_\nu \elm <_\nu \sig$ and $\tau \cong_\nu \pi \sig.$

We enrich the setting of  \cite[Convention 7.10]{QF1} as follows.

\begin{notation}\label{notat:II.7.1} Assume that $\al_1 \neq 0$, $\al_2 \neq 0$, $\al \neq 0$.
We choose $\zt, \eta \in \tT $ with
  \begin{equation}\label{eq:II.7.2}
  \al \cong_\nu \zt \al_1, \qquad \al_2 \cong_\nu \eta \al,
  \end{equation}
and then have
  \begin{equation}\label{eq:II.7.3}
  \eta \zt  \cong_\nu \elm^2.
  \end{equation}
In the important special case  that all  three  parameters $\al_1, \al_2, \al$ are tangible, we take $\zt = \al \al_1^{-1}$, $\eta =  \al \al_2^{-1}$ and have
\begin{equation}
  \al = \zt \al_1, \qquad \al_2 = \eta \al,
\tag{\ref{eq:II.7.2}$'$}\end{equation}
We then further arrange that
\begin{equation}
  \eta \zt =  \elm^2.
\tag{\ref{eq:II.7.3}$'$}\end{equation}


\end{notation}

\begin{remark}\label{rmk:II.7.2} Clearly $\al^2 \nucong \zt \eta^{-1} \al_1 \al_2$. Thus
\begin{equation}\label{eq:II.7.4}
\al_1 \al_2 \nul \al^2 \dss \iff \eta \nul \zt,
\end{equation}
and then $\eta \nul \elm \nul \zt.$

If in addition $R$ is discrete and $\elm \notin \tT$, then $\al_1 \al_2 \nul \al^2$ implies that
\begin{equation}\label{eq:II.7.5}
\eta \nule \tau \nul \elm \nul \sig \nule \zt,
\end{equation}
since $e\tau$, $e \sig$ are now  the elements of $\tG$ nearest to $e \elm \in \tG^{1/2}.$ If even $\al^2 \nug \pi ^{-1} \al_1 \al_2,$ then
\begin{equation}\label{eq:II.7.6}
\eta \nul \tau \nul \elm \nul \sig \nul \zt.
\end{equation}

\end{remark}

\begin{convention}\label{conv:II.7.3} Assuming again that   $\al_1 \neq 0$, $\al_2 \neq 0$, $\al \neq 0$, we distinguish the following subcases of Cases I-III appearing in  \cite[Convention 7.10]{QF1}.
\begin{description}
  \item[Case I] $\al_1 \al_2$ is a $\nu$-square (i.e. $\elm \in \tT$).
  \begin{description}
  \item[\cass{I}{A}] $\al^2 \nug \al_1 \al_2,$ i.e., $\al \nug \elm \al_1.$
  \item[\cass{I}{B}] $\al^2 \nule \al_1 \al_2.$
\end{description} \pSkip

  \item[Case II] $R$ is dense, and $\al_1 \al_2$ is not a $\nu$-square (hence $\elm \notin \tT$).
  \begin{description}
  \item[\cass{II}{A}] $\al^2 \nug \al_1 \al_2,$ i.e., $\al \nug \elm \al_1.$
  \item[\cass{II}{B}] $\al^2 \nul \al_1 \al_2.$
\end{description} \pSkip

 \item[Case III] $R$ is discrete, and $\al_1 \al_2$ is not a $\nu$-square (hence $\elm \notin \tT$).
    \begin{description}
  \item[\cass{III}{A}] $\al^2 \nug \pi^{-1}\al_1 \al_2,$ i.e., $\eta \nul \tau \nul \sig \nul \zt.$
  \item[\cass{III}{B}]  $\al^2 \nucong  \pi^{-1}\al_1 \al_2,$ i.e., $\eta \nucong \tau,  \sig \nucong \zt.$
  \item[\cass{III}{C}] $\al^2 \nul \al_1 \al_2.$
\end{description}

\end{description}

\end{convention}

\begin{prop}\label{prop:II.7.4} Assume that $\al_1, \al_2, \al$ are nonzero. Let $\lm, \mu \in R$, not both zero.
\begin{enumerate} \eroman
  \item In Cases \cass{I}{A}, \cass{II}{A}, \cass{III}{A}
  \begin{equation}
  \label{eq:II.7.7}
  q(\lm x + \mu y ) = \left\{
    \begin{array}{lll}
      \lm^2 \al_1 & &  \lm \nug \zt \mu, \\[1mm]
      e \lm^2 \al_1 = e \lm \mu \al  &  & \lm \nucong \zt \mu, \\[1mm]
      \lm \mu \al  &  & \eta \mu \nul \lm \nul \zt \mu, \\[1mm]
       e \mu^2 \al_2 = e \lm \mu \al  &  &  \lm \nucong \eta \mu, \\[1mm]
      \mu^2 \al_2 &  & \lm \nul \eta \mu. \\
    \end{array}
  \right.
    \end{equation} \pSkip

  \item  In  Case \cass{III}{B} (now $\zt \nucong \sig,$ $\eta \nucong \tau$)

  \begin{equation}\label{eq:II.7.8}
   q(\lm x + \mu y ) = \left\{
    \begin{array}{lll}
      \lm^2 \al_1 & &  \lm \nug \sig \mu, \\[1mm]
      e \lm^2 \al_1 = e \lm \mu \al  &  & \lm \nucong \sig \mu, \\[1mm]
       e \mu^2 \al_2 = e \lm \mu \al  &  &  \lm \nucong \tau \mu, \\[1mm]
      \mu^2 \al_2 &  & \lm \nul \tau \mu. \\
    \end{array}
  \right.
  \end{equation} \pSkip

  \item In  Case  \cass{I}{B} (hence  $\elm \in \tT$)

  \begin{equation}\label{eq:II.7.9}
   q(\lm x + \mu y ) = \left\{
    \begin{array}{lll}
      \lm^2 \al_1 & &  \lm \nug \elm \mu , \\[1mm]
       e \lm^2 \al_1 = e \mu^2 \al_2  &  &  \lm \nucong \elm \mu, \\[1mm]
      \mu^2 \al_2 &  & \lm \nul \elm \mu. \\
    \end{array}
  \right.
  \end{equation} \pSkip
  \item In  Cases  \cass{II}{B},  \cass{III}{C} (hence $\elm \notin \tT$)

  \begin{equation}\label{eq:II.7.10}
   q(\lm x + \mu y ) = \left\{
    \begin{array}{lll}
      \lm^2 \al_1 & &  \lm \nug \elm \mu, \\[1mm]
       \mu^2 \al_2 &  & \lm \nul \elm \mu. \\
    \end{array}
  \right.
  \end{equation} \pSkip
\end{enumerate}

\end{prop}

\begin{proof} In  Cases  \cass{I}{B},  \cass{II}{B},  \cass{III}{C} the form $q$ is quasilinear  an $Rx + Ry,$ as observed in Theorem~\ref{thm:II.5.5}, and hence
$$ q(\lm x + \mu y) = \lm^2 \al_1 + \mu^2 \al_2,$$
and the claims in \eqref{eq:II.7.9}, \eqref{eq:II.7.10} are immediate. In the other cases we have $\al^2 \nug \al_1 \al_2,$ $\al \nucong \zt \al_1,$ and
$$ q(\lm x + \mu y ) = \lm^2 \al_1 + \lm \mu \al + \mu^2 \al_2.$$
Now an easy inspection, which of the three terms on the right are $\nu$-dominant, gives us \eqref{eq:II.7.7} and  \eqref{eq:II.7.8}.
\end{proof}

It remains to handle the degenerate situation where at least one of the parameters $\al_1,$ $\al_2$, and $\al$ is zero.
\begin{convention}\label{conv:II.7.6} We distinguish the following cases, also for later use.
  \begin{description}
   \item[Case IV] $\al_1 \neq 0,$ $\al_2 = 0$, $\al \neq  0.$ \pSkip
  \item[Case V] $\al_1= \al_2 = 0$, $\al \neq 0.$ \pSkip
  \item[Case VI] $\al_1 \neq  0,$ $\al_2 \neq 0$, $\al = 0.$ \pSkip
  \item[Case VII] $\al_1  = \al_2 =  \al = 0.$ \pSkip
  \end{description}
\end{convention}

\begin{notations}\label{notat:II.7.6}
  In Case IV we choose $\zt \in \tT$ with $\al \nucong \zt \al_1.$ In the subcase that both $\al_1, \al $ are tangible we take $\zt = \al \al_1^{-1}$ and then have $\al = \zt \al_1.$
\end{notations}

Notice that the pair $(x,y)$ is excessive in Cases IV, V, while  $q$ is quasilinear on $R x + R y$ in the other two cases.

Now the following is obvious.

\begin{prop}\label{prop:II.7.7}
Let $\lm , \mu \in R,$
 not both zero.

\begin{enumerate} \eroman
  \item In Case IV
  \begin{equation}\label{eq:II.7.11}
   q(\lm x + \mu y ) = \left\{
    \begin{array}{lll}
      \lm^2 \al_1 & \text{if} &  \lm \nug \zt \mu,  \\[1mm]
      e \lm^2 \al_1  = e \lm \mu  \al & \text{if}  &  \lm \nucong \zt \mu , \\[1mm]
       \lm \mu \al  & \text{if}  & \lm \nul \zt \mu . \\
    \end{array}
  \right.
  \end{equation} \pSkip

  \item In Case V

  \begin{equation} \label{eq:II.7.12} q( \lm x + \mu y )  = \lm \mu \al.
  \end{equation} \pSkip

  \item In Case VI \eqref{eq:II.7.9} holds if $\elm \in \tT,$ and \eqref{eq:II.7.10} holds if $\elm \notin \tT$  (as in Cases  \cass{I}{B} resp.  \cass{II}{B},  \cass{III}{C}). \pSkip

  \item In Case VII \ $q(\lm x + \mu y) = 0.$

\end{enumerate}

\end{prop}

\begin{remark}\label{rmk:II.7.8}
The tables in Proposition \ref{prop:II.7.4} and \ref{prop:II.7.7} reveal that (for fixed $x,y$) the $\nu$-value of $q(\lm x + \mu y)$ only depends on the $\nu$-values of $\lm$ and $\mu$.
This is conceptually evident from the equation
$$ eq(\lm x + \mu y) = q ((e\lm)x + (e\mu)y).$$
\end{remark}

We now use these tables to compute the CS-ratios of pairs of vectors in $R x+ R y$ in the case that the pair $(x,y)$ is free and excessive.

\begin{convention}\label{conv:II.7.8}
Assume that the submodule $R x + R y$ of $V$ is free with base $x,y,$ and that the pair $(x,y)$  is excessive. Let $x' , y' \in Rx + Ry$ be given with $x' \neq 0,$ $y' \neq 0$, $x' \neq y'.$  We write
\begin{equation}\label{eq:II.7.13}
x' = \lm_1 x + \mu_1 y, \qquad y' = \lm_2 x + \mu_2 y,
\end{equation}
with $\lm_i, \mu_i \in R.$ We
exclude the (trivial) case that $\tG x' = \tG y'$ and
assume without loss of generality that
\begin{equation}\label{eq:II.7.14}
\lm_1 \mu_2  \ds {>_\nu} \lm_2 \mu_1,
\end{equation}
which for  $\mu_1 \neq 0$ means that $\big[\frac{\lm_1}{\mu_1}\big]_\nu > \big[\frac{\lm_2}{\mu_2}\big]_\nu$. (Recall Definition \ref{defn:II.6.1}.)
Since  the pair $(x,y)$ is free and excessive, the symmetric bilinear form $b$ on $Rx + Ry$ with
\begin{equation}\label{eq:II.7.15}
b(x,x) = b(y,y) = 0, \qquad b(x,y) = \al,
\end{equation}
is a companion of $q | Rx + Ry.$
\end{convention}

\begin{problem}\label{prob:II.7.9} $ $
\begin{enumerate}
  \item[a)] Compute the CS-ratio $\CS(x',y')$ with respect to $(q,b)$ in terms of $\al_1, \al_2, \al$ and the $\lm_i, \mu_j,$ if both $x',y'$ are anisotropic. \pSkip

  \item[b)] Decide which pairs $(x',y' )$ are again excessive.
\end{enumerate}
\end{problem}

We know by Corollary \ref{cor:II.5.7} that in case that $(x',y')$ is \emph{not} excessive, the quadratic  form $q$  is quasilinear on $Rx' + Ry'.$  \emph{We then  say in brief that the pair $(x', y')$ is quasilinear.}

In the following we write $\cong$ instead of $\cong_\nu$ and $\big[\frac{\lm}{\mu} \big]$ instead of $\big[\frac{\lm}{\mu} \big]_\nu$, for short.

As a consequence of \eqref{eq:II.7.14}  and \eqref{eq:II.7.15} we have
\begin{equation}\label{eq:II.7.16}
b(x', y') = \lm_1 \mu_2 \al.
\end{equation}
Since $(x',y')$ is excessive, we are in one of the Cases IA, IIA, IIIA, IIIB, IV, V, and in Case~IIIB if at least one of the elements $\al_1,$ $\al_2$   is tangible (cf. Definition \ref{defn:II.5.6}).

We postpone the degenerate Cases IV and V, and thus assume now that $\al_1 \neq 0,$   $\al_2 \neq 0,$ and $\al^2 \nul \al_1 \al_2.$ We constantly use the table in Proposition \ref{prop:II.7.4}, based on Notation \ref{notat:II.7.1}, and rely heavily  on Theorem \ref{thm:II.5.5}.

Before entering  systematic computations, we warm up with some observations. We have
\begin{equation}\label{eq:II.7.17}
\CS(x,y) = \bigg[ \frac{\al^2}{ \al_1 \al_2} \bigg]  = \bigg[ \frac{\zt^ 2 \al_1^2 }{ \zt \eta \al_1^2} \bigg] = \bigg[ \frac{\zt}{ \eta} \bigg].
\end{equation}
The vectors
\begin{equation}\label{eq:II.7.18}
z := \zt x + y, \qquad w := \eta x + y
\end{equation}
will play a prominent role. We have $$q(z) = e \zt^2 \al_1, \quad q(w) = e \al_2, \quad b(z,w) = \zt \al = \zt^2 \al_1;$$ hence
\begin{equation}\label{eq:II.7.19}
\CS(z,w) = \bigg[ \frac{\zt^4 \al_1^2}{ \zt^2 \al_1 \al_2} \bigg]  = \bigg[ \frac{\zt^ 2 \al_1^2 }{ \zt \eta \al_1^2} \bigg] = \bigg[ \frac{\zt}{ \eta} \bigg] \nug 1.
\end{equation}

We conclude that the pair $(z,w)$ is excessive, except in  Case IIIB.  Then $(z,w)$ is quasilinear, since both $q(z),$ $q(w)$ are ghost. On the other hand
\begin{equation}\label{eq:II.7.20}
\CS(x,z) = \bigg[ \frac{\al^2}{ \al_1 \cdot \zt^2 \al_1} \bigg]  \nucong 1,
\end{equation}
\begin{equation}\label{eq:II.7.21}
\CS(w,y) = \bigg[ \frac{\eta ^2 \al^2}{ \al_1 \cdot \al_2} \bigg]  \nucong 1.
\end{equation}
Thus both pairs $(x,z),$ $(w,y)$ are quasilinear.

The CS-values \eqref{eq:II.7.17} and \eqref{eq:II.7.19}--\eqref{eq:II.7.21} make it plausible that
\begin{equation}\label{eq:II.7.22}
\CS(x',y') \leq \bigg[ \frac{\zt}{\eta} \bigg] = \CS(x,y)
\end{equation}
for all pairs $(x',y')$ in $(Rx + Ry) \sm \{ 0 \} .$  This is indeed true, as we will verify below.

\pSkip

We are ready to compute  $\CS(x', y')$ in all cases.
\begin{enumerate}
  \item[a)] Assume that $\lm_2 \nuge \zt \mu_2$. Now $q(x') \cong \lm_1^2 \al_1$,
  $q(y') \cong \lm_2^2 \al$, and hence
\begin{equation}\label{eq:II.7.23}
\CS(x',y') = \bigg[ \frac{\lm_1 ^2  \mu_2 ^2 \al^2}{ \lm_1^2 \al_1 \lm_2^2 \al } \bigg]  =
\bigg[ \frac{ \mu_2 ^2 \zt ^2}{\lm_2^2 } \bigg] \nule 1.
\end{equation}
 N.B. This is smaller than $1^\nu$ if $\lm_2 \nug \zt \mu_2.$
\pSkip

  \item[b)] Assume that $\lm_1 \nule \eta \mu_1.$ We obtain
\begin{equation}\label{eq:II.7.24}
\CS(x',y') = \bigg[ \frac{\lm_1 ^2  \mu_2 ^2 \al^2}{ \mu_1^2 \al_1 \mu_2^2 \al_2 } \bigg]  =
\bigg[ \frac{ \lm_1 ^2 \al ^2}{\mu_1^2 \al_2^2 } \bigg] =
\bigg[ \frac{ \lm_1 ^2 }{\mu_1^2 \eta^2 } \bigg] \nule 1.
\end{equation}
This can also be deduced from a) by interchanging $x,y$ and $x',y'$.
\pSkip

  \item[c)] Assume that $\lm_1 \nule \zt \mu_1,$ $\lm_2 \nuge \eta \mu_2.$\footnote{Perhaps the most interesting case!}
  Now   $q(x') \cong \lm_1 \mu_1 \al , $  $q(y') \cong \lm_2 \mu_2 \al,$ and hence
\begin{equation}\label{eq:II.7.25}
\CS(x',y') = \bigg[ \frac{\lm_1 ^2  \mu_2 ^2 \al^2}{\lm_1 \lm_2  \mu_1  \mu_2 \al^2 } \bigg]  =
\bigg[ \frac{ \lm_1 \mu_2 }{\mu_1 \lm_2 } \bigg] \nug 1.
\end{equation}
Thus $(x',y')$ is excessive except in Case IIIB. Then there exist no $\nu$-values in $R$ strictly between $\zt$ and $\eta$. Hence $\CS(x', y') = \CS(z,w),$ and we know from the above that $(x',y')$ is quasilinear.
\pSkip

  \item[d)] $\lm_1 \nuge \zt \mu_1$, $\zt \mu_2 \nuge \lm_2 \nuge \eta \mu_2.$  We obtain in the same way
\begin{equation}\label{eq:II.7.26}
\CS(x',y') = \bigg[ \frac{\zt  \mu_2 }{ \lm_2} \bigg]  = \CS(z,y').
\end{equation}
\pSkip

  \item[e)] $\zt \mu_1 \nuge \lm_1 \nuge \eta \mu_1,$ $\lm_2 \nule \eta \mu_2$. We obtain
\begin{equation}\label{eq:II.7.27}
\CS(x',y') = \bigg[ \frac{\lm_1 }{ \mu_1 \eta} \bigg]  = \CS(x',w).
\end{equation}
Again we can also deduce e) from d) by interchanging $x,y$ and $x',y'$.
\pSkip

  \item[f)] The degenerate Case  IV, where $\al_1 \neq 0,$ $\al_2 = 0,$ and only one parameter $\zt$ is present. We obtain
\begin{equation}\label{eq:II.7.28}
\CS(x',y') =
\left\{
  \begin{array}{ll}
    \big[ \frac{  \mu_2 ^2 \zt^2}{ \lm_2^2} \big]  \nule 1  & \text{if } \lm_2 \nuge \zt \mu_2; \\[2mm ]
    \big[ \frac{  \mu_2  \zt}{ \lm_2} \big]  \nug 1  & \text{if } \lm_1 \nuge \zt \mu_1 \text{ and } \lm_2 \nul \zt \mu_2 ; \\[2mm]
    \big[ \frac{  \lm_1 \mu_2}{ \lm_2  \mu_1 } \big]  \nug 1  & \text{if } \lm_1 \nule \zt \mu_1.  \\
  \end{array}
\right.
\end{equation}
\pSkip

  \item[g)] $\al_1 = \al_2 = 0,$  $\al \neq 0$ (Case V). Now
\begin{equation}\label{eq:II.7.29}
\CS(x',y') = \bigg[ \frac{\lm_1 ^2  \mu_2 ^2 \al^2}{\lm_1 \mu_1 \al  \cdot \lm_2    \mu_2 \al  } \bigg]  =
\bigg[ \frac{ \lm_1 \mu_2 }{\mu_1 \lm_2 } \bigg] \nug 1.
\end{equation}
\end{enumerate}
Thus every pair $(x',y')$ is excessive.
\pSkip

Equations \eqref{eq:II.7.28} and \eqref{eq:II.7.29} may be viewed as resulting from \eqref{eq:II.7.23}--\eqref{eq:II.7.27} by putting $\eta = 0$ and $\eta = 0, \zt = \infty,$ respectively. \pSkip

The solution just  obtained for  Problem \ref{prob:II.7.9}.b remains valid if the pair $(x,y)$ is not necessarily free, since we can choose a linear map
$\chi : R \eps_1 + R \eps_2 \to V$  with $\chi(\eps_1) = z$, $\chi(\eps_2) = y$ as in \eqref{eq:II.5.2}, define $\tlq$, $\tlb$ on the free module   $R \eps_1 + R \eps_2$ as in \eqref{eq:II.5.3}, where now $b$ is any companion of $q$ and then apply Proposition \ref{prop:II.5.4}. \{Notice that the value $\al = b(x,y)$ does not depend on the choice of $b$, since $(x,y)$ is excessive.\} We thus arrive at the following theorem.

\begin{thm}\label{thm:II.7.10} Continuing with Notations \ref{notat:II.7.1} and \ref{notat:II.7.6}, we assume that the pair
$(x,y)$ is excessive, and without loss of generality, that either $\al_1 \neq 0$ or $\al_1 = \al_2 = 0.$ Let
$x' = \lm_1 x + \mu_1 y$, $y' = \lm_2 x + \mu_2 y$ with $\lm_i, \mu_i \in R$ and $\lm_1 \mu_2 \nug \lm_2 \mu_1.$
Then $q$ is quasilinear on $Rx' + Ry'$ precisely in the following three cases.
\begin{enumerate}
  \item[1)] $\al_1 \neq 0, \lm_2 \nuge \zt \mu_2;$\pSkip
  \item[2)] $\al_1 \neq 0, \al_2 \neq 0, \lm_1 \nule \eta \mu_1;$ \pSkip
  \item[3)] $\al_1 \neq 0, \al_2 \neq 0,$ $R$ discrete,   $\lm_1 \nucong  \zt \mu_1, \lm_2 \nucong \eta \mu_2.$
\end{enumerate}
Otherwise $(x',y')$ is excessive.
\end{thm}

\section{Supertropicalization: Two examples}\label{sec:II.8}

We illustrate the dependence of the stropicalization $q^\vrp$ of a quadratic form $q: V\to R$ on the choice of a base of the free module $V$ by two examples, which may be regarded as the simplest cases of interest.

We assume that $R$ is a field and $\vrp: R\to U$ is a supervaluation \cite[\S4]{IKR1}. Let $v: R\to M:=eU$ denote the valuation covered by $\vrp.$ Leaving aside a less interesting case, we assume that $v\ne e\vrp,$ i.e., $e\ne 1_U.$ Then $\vrp$ is ``tangible'', i.e., all values $\vrp(a),$ $a\in R,$ are tangible \cite[Proposition~8.13]{IKR1}. Making $U$ smaller we may assume, without loss of generality, that $\vrp(R^*)=\tT,$ $v(R^*)=\tG,$ with $\tT:=\tT(U)$, $\tG=\tG(U).$ Now $U$ is a tangible supersemifield.

We further assume that the supervaluation $\vrp$ is ``tangibly additive'' \cite[Definition~9.6]{IKR1}.\footnote{In general, the tangibly additive supervaluations seem to be the most suitable ones for applications, cf. \cite[\S9-\S11]{IKR1}.} Since $R$ is a ring, even a field, this implies that $\vrp$ is ``very strong'' \cite[\S10]{IKR1}, i.e., for all $a,b\in R,$
\begin{equation}\label{eq:II.8.1}
v(a)< v(b)\ds \Rightarrow\vrp(a+b)=\vrp(b).\end{equation}

We briefly recall the process of stropicalization when  $\dim V=2. $ Let
$$q=(q,v_1,v_2)=\begin{bmatrix} \al_1 & \al\\  &\al_2\end{bmatrix}$$
denote the presentation of the given (functional) quadratic form $q: V\to R$ after choice of a base $v_1,v_2$ of the vector space $V.$ Then
\begin{equation}\label{eq:II.8.2}
q^\vrp:=(q,v_1,v_2)^\vrp=\begin{bmatrix}\vrp(\al_1) & \vrp(\al)\\  & \vrp(\al_2)\end{bmatrix} \end{equation}
is the stropicalization of $q$ with respect to $(v_1,v_2),$ and
\begin{equation}\label{eq:II.8.3}
 b^\vrp=(b,v_1,v_2)^\vrp=\begin{pmatrix} \vrp(2)\vrp(\al_1) & \vrp(\al)\\ \vrp(\al) &\vrp(2)\vrp(\al_2)\end{pmatrix}\end{equation}
is the stropicalization of the unique companion $b: V\times V\to R$ of $q$ with respect to $(v_1,v_2)$, cf. \cite[\eq{9.14} and \eq{9.15}]{QF1}.
Here the presentations
\eqref{eq:II.8.2}, \eqref{eq:II.8.3}
refer to the standard base $\veps_1,\veps_2$ of~$U^2.$

\begin{exampleA*} $ $

$(q,v_1,v_2)=\begin{bmatrix} 0 & 1\\  & 0\end{bmatrix}.$
We take a new base of $V$
$$v_1'=a_{11}v_1+a_{12}v_2,\qquad v_2'=a_{21}v_1+a_{22}v_2,$$
and then have
$$(q,v_1',v_2')=\begin{bmatrix} a_{11}a_{12} & a_{11}a_{22}+a_{12}a_{21}\\
 & a_{21}a_{22}\end{bmatrix}.$$
We abbreviate
$$\tlq:=(q,v_1',v_2')^\vrp,\qquad \tlb:=(b,v_1',v_2')^\vrp.$$
Thus
\begin{equation}\label{eq:II.8.4}
\tlq=\begin{bmatrix} \vrp(a_{11}a_{12})& \vrp(a_{11}a_{22}+a_{12}a_{21})\\  & \vrp(a_{21}a_{22})\end{bmatrix}.
\end{equation}

\begin{description}
  \item[\textbf{Case I}]

 $\vrp(a_{11}a_{22})>_\nu \vrp(a_{12}a_{21}).$

Using \eqref{eq:II.8.1} we obtain from \eqref{eq:II.8.4}
$$\tlq=\begin{bmatrix} \vrp(a_{11}a_{12}) & \vrp(a_{11}a_{22})\\  & \vrp(a_{21}a_{22})\end{bmatrix}.$$
This implies (cf. \eqref{eq:II.8.2}, \eqref{eq:II.8.3})
$$\tlb=\begin{pmatrix} \vrp(2a_{11}a_{12}) & \vrp(a_{11}a_{22})\\
\vrp(a_{11}a_{22})& \vrp(2a_{21}a_{22})\end{pmatrix}.$$
Thus we have
$$\tlq=c\begin{bmatrix} b_1 &1\\  & b_2\end{bmatrix},\qquad \tlb=c\begin{pmatrix}\vrp(2)b_1 & 1\\ 1 & \vrp(2)b_2\end{pmatrix},$$
with
\begin{align*}c&: = \vrp(a_{11}a_{22})\in\tT,\\
 b_1&:=\frac{\vrp(a_{12})}{\vrp(a_{22})}\in\tT\cup\{0\},\\
 b_2&:=\frac{\vrp(a_{21})}{\vrp(a_{11})}\in\tT\cup\{0\},\end{align*}
 and $0<\vrp(2)\le e$ if $\Char R\ne 2,$ while $\vrp(2)=0$ if $\Char R=2.$ The value $\vrp(2)$ will not matter in what follows.

 Notice that $b_1b_2<_\nu 1.$ Thus the pair $(\veps_1,\veps_2)$ is excessive (cf. Definition \ref{defn:II.5.6}). If $b_1b_2\ne0$, we have the CS-ratio
 $$\CS(\veps_1,\veps_2)\cong_\nu\frac{1}{b_1b_2}$$
 with respect to $(\tlq,\tlb).$

\pSkip
 \item[\textbf{Case II}]$\vrp(a_{11}a_{22})<_\nu\vrp(a_{12}a_{21}).$

 Now
 $$\tlq=\begin{bmatrix} \vrp(a_{11}a_{12}) & \vrp(a_{12}a_{21})\\
  & \vrp(a_{21}a_{22})\end{bmatrix},$$
 and we obtain
 $$\tlq=c\begin{bmatrix} b_1 & 1\\  & b_2\end{bmatrix},\qquad\tlb=c\begin{pmatrix} \vrp(2)b_1 & 1\\ 1 & \vrp(2)b_2,\end{pmatrix}$$
 with $c, b_1, b_2$ as above (Case I).
Again $b_1b_2<_\nu 1,$ whence $(\veps_1,\veps_2)$ is excessive and, if $b_1b_2\ne0,$
$$\CS(\veps_1,\veps_2)\cong_\nu \frac{1}{b_1b_2}.$$
\pSkip
\item[\textbf{Case III}] $\vrp(a_{11}a_{22})\cong_\nu \vrp(a_{12}a_{21})\ne0.$

Using the general rule
$$ v(x+y) \ds\leq v(x) + v(y) \quad [= \max\{ v(x), v(y)\}]$$
(cf. \cite[Definition 2.1]{IKR1}) for $m$-valuations, we obtain from \eqref{eq:II.8.4}
$$\tlq=\begin{bmatrix} \vrp(a_{11}a_{12}) & \gamma\\  & \vrp(a_{21}a_{22})\end{bmatrix}$$
with $\gamma\le_\nu\vrp(a_{11}a_{12}),$ and then
$$\tlq=c\begin{bmatrix} b_1 &\delta\\  & b_2\end{bmatrix},\qquad \tlb=c\begin{pmatrix} \vrp(2)b_1 & \delta\\
\delta & \vrp(2)b_2\end{pmatrix}$$
with
$$c=\vrp(a_{11}a_{22})\in\tT,\qquad \delta\le_\nu1,$$
$$b_1=\frac{\vrp(a_{12})}{\vrp(a_{22})}\in\tT,\qquad b_2=\frac{\vrp(a_{21})}{\vrp(a_{11})}\in\tT.$$

Now $\delta^2\le_\nu1=b_1b_2.$ Thus the pair $(\veps_1,\veps_2)$ is quasilinear,\footnote{By this we mean that $\tlq$ is quasilinear on $U\veps_1\times U\veps_2$, hence on  $U\veps_1 +  U\veps_2$, cf. \S\ref{sec:II.5}.}
whence
$$\tlq=[b_1,b_2]:=\begin{bmatrix} b_1 &0\\  & b_2\end{bmatrix}.$$
More precisely, $(\veps_1,\veps_2)$ is weakly CS with respect to $(\tlq,\tlb)$ (cf. Definition~ \ref{def:II.5.10}).

\end{description}
These three cases exhaust all possibilities, since we cannot have $a_{11}a_{22}=a_{12}a_{21}=0,$ because $a_{11}a_{22}-a_{12}a_{21}\ne0.$
\end{exampleA*}

\begin{exampleB*} $ $

$(q,v_1,v_2)=[\al,\bt]:=\begin{bmatrix} \al & 0\\  &\bt\end{bmatrix}$ with $\al\ne0,$ $\bt\ne0.$

We choose a new base
$$v_1'=a_{11}v_1+a_{12} v_2,\qquad v_2'=a_{21}v_1+a_{22}v_2$$
of $V.$ Then
\begin{equation}\label{eq:II.8.5}
(q,v_1',v_2')=\begin{bmatrix} a_{11}^2 \al +a_{12}^2\bt& a_{11}a_{21}\al+a_{12}a_{22}\bt\\
 & a_{21}^2\al+a_{22}^2\bt\end{bmatrix}.
\end{equation}

We use again the abbreviations
$$\tlq:=(q,v_1',v_2')^\vrp,\qquad\tlb:=(b,v_1',v_2')^\vrp.$$

\begin{description}

\item[\textbf{Case I}] $v(a_{11}^2\al)>v(a_{12}^2\bt),\qquad v(a_{21}^2\al)>v(a_{22}^2\bt).$

It follows that $a_{11}\ne0,$ $a_{21}\ne0$, and
$$v(a_{11}a_{21}\al)>v(a_{12}a_{22}\bt).$$
Thus
$$\tlq=\begin{bmatrix}\vrp(a_{11}^2\al)& \vrp(a_{11}a_{21}\al)\\  & \vrp(a_{21}^2\al)
\end{bmatrix}.$$

We conclude that $\CS(\veps_1,\veps_2)=e,$ and hence the pair $(\veps_1,\veps_2)$ is quasilinear (more precisely, weakly CS), whence
\begin{equation}\label{eq:II.8.6}
\tlq=[\vrp(a_{11}^2\al),\vrp(a_{21}^2\al)]=[\vrp(\al),\vrp(\al)].
\end{equation}

\pSkip

\item[\textbf{Case II}] $v(a_{11}^2\al)=v(a_{12}^2\bt),$\qquad $v(a_{21}^2\al)>v(a_{22}^2\bt).$

We have $a_{11}\ne0,$ $a_{12}\ne0,$ since otherwise $a_{11}=a_{12}=0,$ which contradicts $a_{11}a_{22}-a_{12}a_{21}\ne0.$
It follows that
$$v(a_{11}a_{21}\al)>v(a_{12}a_{22}\bt),$$
and we obtain from \eqref{eq:II.8.5}
$$\tlq=\begin{bmatrix}\vrp(a_{11}^2\al + a_{12}^2\bt) & \vrp(a_{11}a_{12}\al)\\  & \vrp(a_{21}^2\al)\end{bmatrix}.$$

\pSkip

\item[\textbf{Case III}] $v(a_{11}^2\al)=v(a_{12}^2\bt),$\qquad $v(a_{21}^2\al) = v(a_{22}^2\bt).$

From $a_{11}a_{22}-a_{21}a_{21}\ne0$ we conclude that all entries $a_{ij}\ne 0.$
We cannot say more in this generality.
\pSkip

\item[Note]

In Cases II and III the nature of the pair $(\veps_1,\veps_2)$ (excessive, quasilinear, weakly CS, \dots) remains undetermined by the values $v(a_{ij}).$ This may change  if we have specified information about the value $v$ and $\al$ and $\bt.$ For example, if $v$ is compatible with a total ordering $\leq$ of the field $R$, cf.  \cite{lam}, and $\al > 0 $, $\bt > 0$, then
$$ v(a_{11}^2\al + a_{12}^2 \bt) = \max \{   v(a_{11}^2\al) , v( a_{12}^2 \bt)\} $$
 and we can say more. It may well happen that $\tlq$ is not quasilinear. (We do not go into details.)

Notice that Cases II and III can only occur if $v(\al),$ $v(\bt)$ are square equivalent (cf. \Qdefref{7.1}).

\pSkip

\item[\textbf{Case IV}] $v(a_{11}^2\al) > v(a_{12}^2\bt),$\qquad $v(a_{21}^2\al) < v(a_{22}^2\bt).$

Now we read off from \eqref{eq:II.8.5} that
$$\tlq=\begin{bmatrix} \vrp(a_{11}^2\al) & \vrp(a_{11}a_{21}\al+a_{12}a_{22}\bt)\\
 & \vrp(a_{22}^2\bt)\end{bmatrix}.$$

We have $a_{11}\ne0,$ $a_{22}\ne0.$ Thus the CS-ratio $\CS(\veps_1,\veps_2)$ exists and
\begin{align*}
\CS(\veps_1,\veps_2)&=\frac{v(a_{11}a_{21}\al+a_{12}a_{22}\bt)^2}{v(a_{11}^2a_{22}^2\al\bt)}\\[1mm]
& \leq \frac{v(a_{11}^2a_{21}^2\al^2)}{v(a_{11}^2a_{22}^2\al\bt)}+\frac{v(a_{12}^2a_{22}^2\bt^2)}{v(a_{11}^2a_{22}^2\al\bt)}\\[1mm]
&=\frac{v(a_{21}^2\al)}{v(a_{22}^2\bt)}+\frac{v(a_{12}^2\bt)}{v(a_{11}^2\al)} \ds <e.\end{align*}
Thus the pair $(\veps_1,\veps_2)$ is CS (cf. Definition  \ref{def:II.5.10}) and
$$\tlq=[\vrp(a_{11}^2\al),\vrp(a_{22}^2\bt)]=[\vrp(\al),\vrp(\bt)].$$

\item[\textbf{More Cases}]

If $$v(a_{11}^2\al)\le v(a_{12}^2\bt),\qquad v(a_{21}^2\al)>v(a_{22}^2\bt),$$
then interchanging $v_1,v_2$ we come back to Cases I, II.
\pSkip

If $$v(a_{11}^2\al)<v(a_{12}^2\bt),\qquad v(\al_{21}^2\al)\le v(a_{22}^2\bt),$$
then interchanging also $v_1',v_2'$ we come again  back to Cases I,II.

\pSkip

Finally, if
$$v(a_{11}^2\al)<v(a_{12}^2\bt),\qquad v(a_{21}^2\al)>v(a_{22}^2\bt),$$
then interchanging $v_1,v_2$ we come  back to Case IV.

\end{description}
Thus Cases I--IV exhaust all possibilities up to interchanging $v_1,v_2$ and/or $v_1',v_2'.$

This completes Example B.
\end{exampleB*}

If the values $v(\al),v(\bt)$ are \textit{not} square equivalent, then Case IV in Example B does not occur, as observed above. Thus we may state

\begin{prop}\label{prop:II.8.1}
Assume that $R$ is a field, and that $\vrp:R\to U$ is a tangibly additive supervaluation which is not ghost, and hence is very strong. Let $q=[\al,\bt]$ be a binary diagonal form over $R$ with $v(\al),v(\bt)$ not square equivalent $(v:=e\vrp).$ Then all stropicalizations of $q$ by $\vrp$ are quasilinear.
\end{prop}

\begin{remark}\label{rem:II.8.2}
This proposition does not contradict Example A. Assume that $\Char R\ne2.$ If $q=\left[\begin{smallmatrix}0 & 1\\  & 0\end{smallmatrix}\right]$ and $q'=[\al,\bt]$ are forms over $R$ with $\al\ne0,$ $\bt\ne0,$ and $q\cong q',$ then $\bt=-\lm^2\al$ for some $\lm\in R^*,$ and hence $v(\al)$ and $v(\bt)$ are square equivalent.
\end{remark}

\section{The minimal ordering on a free $R$-module}\label{sec:I.6}

In this section $R$ is a supertropical semiring.  If $V$ is any module over $R$, we define on $V$ a binary relation~$\le_V$ as follows: \\ For any $x,y\in U$,
 \begin{equation}\label{eq:I.6.1}
x\le_Vy \dss \rightleftharpoons\exists z\in V: x+z=y.
\end{equation}

This relation is clearly reflexive $(x\le x)$ and transitive $(x\le y, y\le z\Rightarrow x\le z).$ It is also antisymmetric, \textit{hence is a partial ordering on the set} $V.$ Indeed, assume that $x + z=y$ and $y+w=x.$ This implies $x+z+w=x,$ $y+z+w=y,$ and then
$$x+e(z+w)=x,\quad y+e(z+w)=y.$$
Adding $z$ at both sides of the first equation, and using that $z+ez=ez$, we obtain
$$y=x+e(z+w)=x,$$
as desired.

Clearly, our partial ordering $\le_V$ satisfies  the rules $(x,y,z\in V)$
 \begin{gather}
 0\le z,\label{eq:I.6.2}\\
 x\le y \dss \Rightarrow x+z\le y+z.\label{eq:I.6.3}
 \end{gather}
 (Thus, $x\le y,$ $x'\le y'\Rightarrow x+x'\le y+y'.)$
 It is now obvious that any partial ordering $\le'$ on~ $V$ with the properties \eqref{eq:I.6.2}, \eqref{eq:I.6.3}, is a refinement of $\le_V:$ If $x\le_Vy,$ then $x\le' y.$

\begin{defn}\label{defn:I.6.1}
We call $\le_V$ the \bfem{minimal ordering} on the $R$-module $V.$ \footnote{In the special case $V= R$ the minimal ordering has been discussed already in \Qsecref{5}, including an explanation of the term ``minimal''.}
\end{defn}
\begin{notation} \label{notat:I.6.2}
As long as no other orderings of $V$ come into play, we usually write $x\le y$ instead of $x\le_V y.$ But notice that  if $W$ is a submodule of $V,$ it may happen for $x,y\in W$ that $x\le_V y$ but not $x\le_Wy.$

As usual, $x<y$ means that $x\le y$ and $x\ne y.$\end{notation}

In particular, $R$ itself carries the minimal ordering $\le_R.$ It already showed up in \cite[Proposition 11.8]{IKR1} and \Qsecref{5}. Again, we usually write $\lm \le \mu$ instead of $\lm \le_R\mu.$

Scalar multiplication is compatible with these orderings on $R$ and $V:$
\begin{equation}\label{eq:I.6.4}
\lm \le\mu,\ x\le y\dss \Rightarrow \lm  x\le \mu y
\end{equation}
for all $\lm ,\mu\in \R,\ x,y\in V.$

Before moving on to details about minimal orderings, we hasten to point out that these orderings are relevant for the geometry in a supertropical quadratic space. This is apparent already from the definition of quadratic forms \Qdefref{0.1}.

\begin{remark}\label{rem:I.6.3} As before, let $V$ be a module over a supertropical semiring $R.$ If $(q,b)$ is a quadratic pair on $V,$ then for all $x,y,z,w\in V$ the following hold:
\begin{equation}\label{eq:I.6.5}
\qquad x\le _V z \dss \Rightarrow q(x)\le_R q(z),\end{equation}
\begin{equation}\label{eq:I.6.6}
x\le _V z,\ y\le_V w\dss \Rightarrow b(x,y)\le_R b(z,w),\end{equation}
\begin{equation}\label{eq:I.6.7}
b(x,y) \ds{\le_R}  q(x+y).\end{equation}
\end{remark}

The minimal ordering of $R$ has the following detailed description in terms of the $\nu$-dominance relation and the sets $eR$ and $\tT=R\sm (eR).$

\begin{prop}\label{prop:I.6.3}
\quad

\begin{enumerate}
\item[a)] Assume that $x\in eR.$ Then $x$ is comparable (in the minimal ordering) to every $y\in R.$ More
precisely, using the $\nu$-notation,
\begin{equation}\label{eq:I.6.8}
x<y \dss \Leftrightarrow x<_\nu y,\end{equation}
\begin{equation}\label{eq:I.6.9}
y<x \dss \Leftrightarrow \text{either}\ y<_\nu x\ ,  \text{or}\ y\in\tT\ \text{and}\ y\cong_\nu x.\end{equation} \pSkip
\item[b)] Assume that $x\in\tT,$ $y\in R.$ Then
\begin{equation}\label{eq:I.6.10}
x<y\dss \Leftrightarrow \text{either}\ x <_\nu y\ , \text{or}\ x\cong_\nu y\ \text{and}\ y\in eR,\end{equation}
\begin{equation}\label{eq:I.6.11}
y<x \dss \Leftrightarrow   y<_\nu x.\end{equation}
Thus $x$ and $y$ are incomparable iff $y\in \tT$ and $x\ne y,$ but $x\cong_\nu y.$
\end{enumerate}

\end{prop}

\begin{proof}
All this can be read off from the description \eqref{eq:0.6} of the sum $x+y$ of $x,y\in R$ in terms of the $\nu$-dominance relation, recalled from \cite[\S2]{IzhakianRowen2007SuperTropical}.\footnote{The general assumption in \cite{IzhakianRowen2007SuperTropical}, that the monoid $(eR,\cdot \; )$ is cancellative, is not needed here. It is only relevant if products $xy$ are involved.}
\end{proof}

From Proposition  \ref{prop:I.6.3} we read of that for any two elements $x,y$ of $R$ the maximum $x \vee y:= \max_R\{ x,y\} $ exists, namely
\begin{equation}\label{eq:I.6.13}
x \vee y\dss  = \left\{
                  \begin{array}{ll}
                    x & \hbox{if \ } ex < ey; \\
                    y & \hbox{if \ } ex > ey;\\
                    ex & \hbox{if \ } ex =  ey.
                  \end{array}
                \right.
\end{equation}
Note that
\begin{equation}\label{eq:I.6.13}
e(x \vee y) = (ex)\vee (ey) = ex + ey,
\end{equation}
while  for arbitrary $\lm \in R$ in general only
$\lm(x \vee y) \leq (\lm x)\vee (\lm y)$, but here we have equality if $R$
 is a supersemifield.

Assume now that $V$ is a free $R$-module with base $(\veps_i \ds |i\in I).$ If $x,y$ are vectors in $V$ with coordinates $(x_i\ds |i\in I)$, $(y_i \ds |i\in I),$ i.e.,
$$x=\sum_{i\in I}x_i\veps_i\quad y=\sum_{i\in I}y_i\veps_i,$$
where $x_i\ne0$ or $y_i\ne0$ only for finitely many $i\in I, $ then clearly
\begin{equation}\label{eq:I.6.14}
x\le_Vy \dss \Leftrightarrow \forall i\in I\quad x_i\le_Ry_i.\end{equation}

%
%
%
%
%

Moreover, the maximum $x \vee y = \max_V\{x,y\}$, exists, and
\begin{equation}\label{eq:I.6.15}
x\vee y \dss =  \sum _{i \in I} (x_i \vee y_i)\veps_i.\end{equation}
It will be helpful below to argue by use of the \textbf{support} of an element
$x = \sum_{i\in I} x_i\veps_i$ of the free module $V$ defined as follows
\begin{equation}\label{eq:I.6.16}
\supp(x) := \{  i \in I \ds | x_i \neq 0\}.    \end{equation}
As consequence of \eqref{eq:I.6.15} we have
\begin{equation}\label{eq:I.6.17}
\supp(x \vee y) =\supp(x) \cup \supp(y).    \end{equation}

Notice that   $ \supp(x)$ is essentially independent of the choice of the base
$(\veps_i  \|i\in I)$, since up to permutation every other base of $V$ arises by multiplying the $\veps_i$ by units of $R$ \Qthmref{0.9}. Notice also that $\supp(x)$ is empty iff $x= 0$, and that $y \leq x$ implies $\supp(y) \subseteq \supp(x)$.

\section{$q$-minimal vectors with small support}\label{sec:II.9}

In this section $R$ is again a  supertropical semiring. In all $R$-modules we work with their minimal orderings.

\begin{defn}\label{defn:II.9.1} $ $
\begin{enumerate}\item[a)] We call a map $\map  :V\to W$ between $R$-modules $V,W $ \bfem{monotonic}
if for any $x,y\in V$
$$y\le x \dss \Rightarrow \map  (y)\le\map  (x).$$ \pSkip
\item[b)] Given a monotonic map $\map  : V\to W$, we call a vector $x\in V$ $\map  $-\bfem{minimal},
 if there \bfem{does
not} exist a vector $x'<x$ in $V$ with $\map  (x')=\map  (x).$
\end{enumerate}\end{defn}

\begin{examples}\label{examps:II.9.2} $ $ \begin{enumerate}
\item[i)] For any $n\in\N$ and $c\in R$, the map $R\to R,$ $x\mapsto cx^n,$ is additive, and hence monotonic. More generally,
every monomial map $R^n\to R,$
$$(x_1,\dots, x_n) \ds \mapsto cx_1^{\al _1}\cdots x_n^{\al _n}, \qquad (\al _i\in\N_0),$$
is monotonic, and hence every polynomial map $f: R^n\to R$ is monotonic.
\pSkip

\item[ii)] Every quadratic form $q: V\to R$ on an $R$-module $V$ is monotonic, cf.
Remark \ref{rem:I.6.3}. We note the trivial fact
that an isotropic vector $x\in V\sm\{0\}$ is never $q$-minimal, since $0<x,$ but $q(x)=q(0)=0.$

\end{enumerate}
\end{examples}

Given a quadratic form $q:V\to R$, we turn to the problem of determining the $q$-minimal vectors in $V$ in the case that the $R$-module
$V$ is free, and, if possible, at later stages also in more general situations. The following distinction of the
vectors in $V$ will be useful here and elsewhere.

\begin{defn}\label{defn:II.9.3}
We call a vector $x\in V\sm\{0\}$ $\bf g$-\bfem{isotropic}, if $q(x)\in eR,$ and we call $x$ $\bf g$-\bfem{anisotropic},
if $q(x)\in\tT.$\footnote{The letter ``g'' alludes to ``ghost''.} The zero vector is regarded as both $g$-isotropic
and $g$-anisotropic.\end{defn}

%

\begin{prop}\label{prop:II.9.4}
Assume that $V$ is free with base $(\veps_i  \|i\in I).$ Let $x\in V\sm\{0\}$ be $q$-minimal.
Then $|\supp(x)|\le 2$ if
$q(x)\in\tT$, and $|\supp(x)|\le 4$ if $q(x)\in\tG.$
\end{prop}

\begin{proof}
We have a finite non-empty subset $ J =\supp(x)$ of $I$, such that $x=\sum\limits_{i\in J }x _i
\veps_i,$ all $x _i\ne0.$ We choose a companion $b$ of $q.$ Then
\begin{equation}
q(x)=\sum_{i\in J }x _i^2q(\veps_i)+\sum_{\substack{i<j\\i,j\in J }}x _ix _j
b(\veps_i,\veps_j).\tag{$\ast$}
\end{equation}
and $q(x)\ne0.$

If $q(x)\in \tT ,$ the sum on the right of $(\ast$) contains  a unique $\nu$-dominant term. If this term
is $x ^2_kq(\eps_k),$ then $x _k\eps_k\le x$ and $q(x _k\eps_k)=q(x);$ hence $x=x _k\eps_k$ and $ J
=\{k\}.$ If the $\nu$-dominant term is $x _kx _\ell b(\veps_k,\veps_\ell),$ then $x _k
\eps_k+x _\ell \veps_\ell\le x$ and again both vectors have the same $q$-values, and hence $x=x _k\eps_k+x _\ell
\eps_k,$ and $ J =\{k,\ell\}.$ Indeed, then
$$ q(x) = x_k x_\ell b(\eps_k, \eps_\ell) \leq q(x_k \eps_k +  x_\ell \eps_\ell) \leq q(x).$$

If $q(x)\in\tG$, then on the right of $(\ast$) there exists either a $\nu$-dominant term, which is ghost,
or there exist two $\nu$-dominant terms which are tangible. In the first case, we see as above that $| J |
\le 2,$ and in the second that $| J |\le 4.$
\end{proof}


\begin{cor}\label{cor:II.9.5}
Assume in \propref{prop:II.9.4} also that $q$ is quasilinear. Then $|\supp (x)|=1$ if $q(x)\in \tT$, and $|\supp (x)|
\le 2$ if $q(x)\in\tG.$\end{cor}

\begin{proof}
We choose the companion $b=0.$ Now, in the above  arguments   no $\nu$-dominant terms $x _kx _\ell b(
\veps_k,\veps_\ell)$ show up.\end{proof}

Recall from the last lines of \S\ref{sec:I.6} that for  vectors $x',x$ in $V$ with $x' \leq x$ the support of $x'$ is contained in the support of $x$.
Thus in searching for $q$-minimal vectors in $V$ it is not loss of generality to assume that
$|I| \leq 4$. If $q$ is quasilinear  we may even assume that $|I| \leq 2$.

We now deal with the case that $|I| \leq 2$, postponing the cases  $|I| = 3$ and $|I| = 4$ to the next section.

\begin{prop}\label{prop:II.9.6}
{}\quad

\begin{enumerate}
\item[a)] Assume that $V$ is free with a single  base vector $\veps_1.$ When $q(\veps_1)\in\tT,$ all
vectors in $V$ are $q$-minimal. If $q(\veps_1)\in\tG,$ a vector $\lm \veps_1$ is $q$-minimal iff
$\lm \in\tT.$ \pSkip

\item[b)] Assume that $V$ is free with base $(\veps_1,\veps_2)$, and that $q$ is quasilinear. Let
$\al _1:=q(\veps_1),$ $\al _2:=q(\veps_2).$ A vector $x=\lm \veps_1+\mu\veps_2$ with
$\lm ,\mu\ne0$ is $q$-minimal iff $\lm ,\mu,\al _1,\al _2\in\tT$ and $\lm ^2\al _1\cong_\nu \mu^2
\al _2.$ (Thus every $q$-minimal vector with $|\supp (x)|=2$ is $g$-isotropic.)
\end{enumerate}
\end{prop}

\begin{proof}
a): Let $\al _1:=q(\veps_1)$ and $x:=\lm \veps_1\in V.$ We have $q(x)=\lm ^2\al _1.$ Assume
first that $\al _1\in\tT.$ If $x'=\lm '\veps_1$ is a second vector, then $x'<x$ iff $\lm '<\lm $
iff $\lm '{}^2\al _1<\lm ^2\al _1.$ Thus $x$ is $q$-minimal. Assume now that $\al _1\in\tG.$ If
$\lm \in\tG,$ there exists $\lm '\in\tT$ with $\lm '\cong_\nu\lm ,$ and then $\lm '<
\lm .$ For $x'=\lm '\veps_1$ we have $x'<x,$ but $q(x')=\lm '{}^2\al _1=\lm ^2\al _1=q(x).$
Thus $x$ is not $q$-minimal. If $\lm \in\tT$ and $\lm '<\lm $ then $\lm '<_\nu\lm $ (cf.
\eqref{eq:I.6.11}); hence $$q(x')=\lm '{}^2\al _1<_\nu\lm ^2\al _1=q(x),$$ and a fortiori $q(x')<q(x).$ Thus $x$
is $q$-minimal.
\pSkip

 b): We have $q(x)=\lm ^2\al _1+\mu^2\al _2.$ If $q(x)=0,$ then $x$ is not $q$-minimal (cf. Example \ref{examps:II.9.2}.ii).

Assume now that $q(x)\ne0.$ If $\lm ^2\al _1<_\nu\mu^2\al _2$ then $q(x)=\mu^2\al _2=q(\mu\veps_2),$
and $x$ is not $q$-minimal, ditto if $\lm ^2\al _1>_\nu\mu^2\al _2.$ Assume henceforth that $\lm ^2\al _1
\cong_\nu\mu^2\al _2.$ Then $q(x)\in\tG$ and $\al _1\ne0,$ $\al _2\ne0.$ If $\lm ^2\al _1$ or
$\mu^2\al _2$ is ghost, then $q(x)=q(\lm \veps_1),$ resp. $q(x)=q(\mu\veps_2),$ and thus $x$ is not
$q$-minimal. We are left with the case that both $\lm ^2\al _1,$ $\mu^2\al _2$ are tangible. This means that
$\lm ,\mu,\al _1,\al _2\in \tT.$

If $x'<x,$ then either $x'\le \lm '\veps_1+\mu\veps_2$ or $x'\le \lm \veps_1+\mu'\veps_2$
with $\lm '<\lm $, resp. $\mu'<\mu.$ In the first case, $\lm '<_\nu\lm  $ (cf.  \eqref{eq:I.6.11}),
hence $\lm '{}^2\al _1<_\nu\lm ^2\al _1\cong_\nu\mu^2\al _2,$ and
$$q(x')\leq q(\lm '\veps_1+\mu\veps_2)=\mu^2\al _2<e\mu^2\al _2=q(x).$$
In the second case, $q(x')<q(x)$ for the same reason. Thus $x$ is $q$-minimal.
\end{proof}

\emph{We now assume that $\tG$ is a cancellative monoid under multiplication and $\tG= e \tT$}, furthermore that $(q,b)$ is a quadratic pair on the free binary module
$V:=R\veps_1+R\veps_2.$ We search for all $q$-minimal vectors in $V$ with full support.

Let $\al _1:=q(\veps_1),$ $\al _2:=q(\veps_2),$ $\bt :=b(\veps_1,\veps_2),$ and $x = x_1 \veps_1 + x_2 \veps_2$ with $x_1 \neq 0$, $x_2 \neq 0$. Then
\begin{equation}\label{eq:sstr}
q(x) = \al_1 x_1^2 + \bt x_1 x_2 + \al_2 x_2^2.
\tag{$\ast \ast$}
\end{equation}
Looking at the $\nu$-dominant terms in the sum \eqref{eq:sstr} we will run through several cases and will easily find out when $x$ is $q$-minimal.

\begin{enumerate} \dispace
\item[0)] Assume that $\al_1 x_1^2$ (or $\al_2 x_2^2$) is the only $\nu$-dominant term. Then $q(x) = q(x_1  \veps_1)$ or $q(x) = q(x_2  \veps_2)$. Clearly $x$  is not $q$-minimal.

\item[1)] Assume that both  $\al_1 x_1^2$ and  $\al_2 x_2^2$ are $\nu$-dominant.
If, say,  $\al_1 x_1^2$ is ghost, then $q(x) = q(x_1  \veps_1)$ again, and $x$  is not $q$-minimal.
If  both  $\al_1 x_1^2$ and  $\al_2 x_2^2$ are tangible, then for a vector
$ x' = x'_1 \veps_1 + x'_2 \veps_2 < x$   either $x'_1 < x_1$ or $x'_2 < x_2$, which implies
$x'_1 <_\nu x_1$ or $x'_2 <_\nu x_2$, since both $x_1', x'_2$ are tangible. We conclude that
$q(x') < q(x)$. Thus $x$ is $q$-minimal iff $\al_1, \al_2, x_1, x_2$ are all tangible.

\item[2)] Assume that  $\al_1 x_1^2 \nucong \bt x_1 x_2 > \al_2 x_2^2$. Then $q(x) = e \al_1 x_1^2 = e \bt x_1 x_2 \in \tG$. If $\al_1 x_1^2 \in \tG$, then choosing $x'_1\in \tT$ with $e x'_1 = x_1$ we obtain a vector $ x' = x'_1 \veps_1 + x_2 \veps_2  < x$ with $q(x') = \al_1 'x_1^2 + \bt x'_1 x_2 = q(x)$, and so $x$ is not $q$-minimal.

Assume now that $\al_1 x_1^2 \in \tT$. If $ x' = x'_1 \veps_1 + x'_2 \veps_2  < x$, then either
$x'_1 < x_1$, $x'_2 \leq x_2$, or $x'_1 = x_1$, $x'_2 < x_2$.
If $x'_1 < x_1$, then $x'_1 <_\nu x_1$, whence
$\al_1 x_1'^2 <_\nu \al_1 x_1^2$  $\bt x'_1 x_2 <_\nu \bt x_1 x_2  $, and we see that $q(x')< q(x)$. But if
$x'_1 = x_1$, $x'_2 < x_2$, $ex'_2 = x_2$, and $\bt \in \tG$, then
$q(x') = q(x)$, while if $\bt \in \tT$ this cannot happen. We conclude that $x$ is $q$-minimal iff $\al_1, \bt, x_1$ are all tangible.

\item[3)] Analogously, if $\al_2 x_2^2 \nucong \bt x_1 x_2 > \al_1 x_1^2$, then $x$ is $q$-minimal iff
$\al_2, \bt, x_2$ are all tangible.

\item[4)] Assume that  $\al_1 x_1^2 <_\nu \bt x_1 x_2$, $\al_2 x_2^2 <_\nu \bt x_1 x_2$. Now  $q(x) = \bt x_1 x_2 .$
Arguing similarly as in Case 3), we see that, when $\bt \in \tG$ then $x$ is
 $q$-minimal iff $x_1 \in \tT$ and $x_2 \in \tT$, while when $\bt \in \tT$, then $x$ is
 $q$-minimal iff $x_1 \in \tT$ or $x_2 \in \tT$. Thus all together  $x$ is
 $q$-minimal iff at most one of the elements $\bt, x_1, x_2$ is ghost.
\end{enumerate}

Summarizing we obtain
\begin{thm}\label{thm:II.9.7}
Assume that $V$ is free with base $\veps_1, \veps_2$ and $x = x_1 \veps_1 +  x_2 \veps_2$ with $x_1 \neq 0$,  $x_2 \neq 0$.
Let  $q = \left[\begin{smallmatrix} \al_1 & \bt \\  &\al_2\end{smallmatrix}\right]$. Then  $x$ is $q$-minimal exactly in the following cases:

\begin{enumerate} \dispace
\item[1)] $\al_1 x_1^2 \nucong  \al_2 x_2^2 \geq_\nu  \bt x_1 x_2$ and  $\al_1, \al_2, x_1, x_2 \in \tT$;
\item[2)] $\al_1 x_1^2 \nucong  \bt x_1 x_2 >_\nu \al_2 x_2^2   $ and  $\al_1, \bt, x_1 \in \tT$;
\item[3)] $\al_2 x_2^2 \nucong  \bt x_1 x_2 >_\nu \al_1 x_1^2   $ and  $\al_2, \bt, x_2 \in \tT$;
\item[4)] $\bt x_1 x_2 >_\nu \al_1 x_1^2 +    \al_2 x_2^2   $ and  at most one of the elements $ \bt, x_1, x_2$ is ghost.
\end{enumerate}
\end{thm}

\begin{comment} In order to clarify the situation observe that in Cases 2)-4) we have
$\al_1 x_1^2 \cdot \al_2 x_2^2 < _\nu  (\bt x_1 x_2)^2$, whence $\al_1 \al_2 < _ \nu \bt ^2$, while in Case 1)
we have $\bt ^2 <_\nu \al_1 \al_2$. Thus $(\veps_1, \veps_2)$ is excessive w.r. to $q$ in Cases 2)-4), but quasilinear in Case 1).
\end{comment}

Concerning  $g$-anisotropic vectors we note the following immediate consequence of Theorem~\ref{thm:II.9.7}.
\begin{cor}\label{cor:II.9.8}
Assume again that  $x = x_1 \veps_1 +  x_2 \veps_2$ and $q = \left[\begin{smallmatrix} \al_1 & \bt \\  &\al_2\end{smallmatrix}\right]$. Then  $x$ is $q$-minimal and $g$-anisotropic iff $\bt, x_1, x_2$ are tangible and
$\al_1^2 x_1^2 + \al_2^2 x_2^2 < _\nu \bt x_1 x_2.$
\end{cor}

Returning to the tables of $q$-values in \S\ref{sec:II.7} it is of interest to ask which of the vectors
$\lm \veps_1 + \mu \veps_2$ there are $q$-minimal. We only consider the case that $\al^2 > _\nu \al_1 \al_2$ in the notations used there, since otherwise $q$ is quasilinear and the matter is settled by Proposition \ref{prop:II.9.6}.b.

\begin{thm}\label{thm:II.9.9}
Assume that $R$ is a nontrivial tangible supersemifield, and $q$ is a quadratic form on the free binary $R$-module
$V=R\veps_1+R\veps_2.$ Let $b$ be a companion of $q$, and assume that $\al _1\al _2<_\nu\al ^2$ with
$\al _1:=q(\veps_1),$ $\al _2:=q(\veps_2),$ $\al :=b(\veps_1,\veps_2).$ We use Notations
\ref{notat:II.7.1} and \ref{notat:II.7.6}. Let $x=\lm \veps_1+\mu\veps_2$ with $\lm ,\mu\ne0.$
\begin{enumerate}
\item[i)] If $\al _1\ne0,$ $\al _2\ne0$, then $x$ is $q$-minimal iff either $\lm \cong_\nu\zeta\mu$ and
$\al _1, \lm \in\tT$, or $\lm \cong_\nu\eta\mu$ and $\al _2 ,\mu\in\tT$, or
$\eta\mu<_\nu\lm <_\nu\zeta\mu$ and at most one of the three elements $\al ,\lm ,\mu$ is ghost
\pSkip

\item[ii)] If $\al _1\ne0$, $\al _2=0,$ then $x$ is $q$-minimal iff either $\lm \cong_\nu\zeta\mu$ and
$\al _1,\lm \in\tT,$ or $\lm <_\nu\zeta\mu$ and at most one of the elements $\al ,\lm $ is
ghost.
\pSkip
\item[iii)] If $\al _1=\al _2=0,$ then $x$ is $q$-minimal iff at most one of the elements $\al ,\lm ,\mu$ is
ghost.
\end{enumerate}
\end{thm}

\begin{proof} Browse through tables \eqref{eq:II.7.7}, \eqref{eq:II.7.8}, \eqref{eq:II.7.11}, \eqref{eq:II.7.12} and apply Theorem \ref{thm:II.9.7}, reading $\lm, \mu, \al$ for $x_1, x_2, \bt$.
\end{proof}



%
%
%

\section{$q$-minimal vectors with big support}\label{sec:7}

Again we assume that $R$ is a tangible supertropical semiring, $\tG$ is a cancellative monoid,  $V$ is
a free $R$-module with base $(\veps_i \ds|i\in I)$,  and $q:V\to
R$ is a quadratic form. For later use, we adopt the following
notation.

\begin{notation}\label{notation:7.1}
Let $x=\sum\limits_{i\in I} x_i\veps_i\in V$ and $ J $ a subset of $I.$ We put
$$x( J ):=\sum_{i\in  J } x_i\veps_i.$$
If $ J =\{i\}$ or $ J =\{i,j\},$ $i\ne j,$ we
 write for short $x(i)$ or $x(i,j)$ instead of $x(\{i\})$, $x(\{i,j
\}).$
\end{notation}

Assume now that $I=\{1,\dots,n\}$ with $n=3$ or $n=4$, and that
$x\in V$ is a vector of full support,
$$x=\sum_{i=1}^n x_i\veps_i,\quad\text{all}\quad x_i\ne0.$$
We choose a companion $b$ of $q,$ and then have a presentation
\begin{equation}\label{eq:7.1}
q(x)=\sum_{i=1}^n\al_ix_i^2+\sum_{i<j}\bt_{ij}x_ix_j.
\end{equation}
We ask, under which conditions is $x$ $q$-minimal, and then search
for possibilities to write $x$ as the supremum $y\vee z$ of two
$q$-minimal vectors $y,z\in V$ of small support, i.e., $|\supp
(y) |\le 2,$ $|\supp (z)|\le 2.$

As in \S\ref{sec:II.9}, we look for the $\nu$-dominant terms in the
sum \eqref{eq:7.1}. If there is only one dominant term,
$\al_ix_i^2$ or $\bt_{ij}x_ix_j,$ then $q(x)=q(x(i))$ or
$q(x)=q(x(i,j)),$ and so $x$ is not $q$-minimal. Henceforth, we
assume always that there are at least two dominant terms, and so
$q(x)\in\tG.$ Furthermore,  we assume that all $\nu$-dominant terms
are tangible, since otherwise again $q(x)=q(x( J ))$ for some
$ J \varsubsetneqq I.$

We first study the case $n=3 $ and run through several subcases,
as follows:

\begin{enumerate}
\item[A)] Assume that in \eqref{eq:7.1} there occurs a
$\nu$-dominant term $\al_ix_i^2.$ Then, if $x$ is $q$-minimal
there is exactly one further dominant terms $\bt_{jk}x_jx_k$ and
$(i,j,k)$ is a permutation of $(1,2,3), $ since otherwise again
$q(x)=q(x(J))$ for some $ J \varsubsetneqq I.$ We have
$$x=x(i)\vee x(j,k),$$
and $q(x(i))=\al_ix_i^2\in \tT,$
$$q(x(j,k))=\al_jx_j^2+\bt_{ik}x_jx_k+\al_kx_k^2\in \tT.$$
It follows that
$$\al_jx_j^2+\al_kx_k^2<_\nu \bt_{jk}x_jx_k,$$
and we read off from \thmref{thm:II.9.7}  that
$x(j,k)$ is $q$-minimal. By \propref{prop:II.9.6}.a also $x(i)$ is
$q$-minimal.
\end{enumerate}

Note furthermore that
$$b(x(i),x(j,k))<_\nu q(x(i)) \cong_\nu q(x).$$

Assume now that all $\nu$-dominant terms in the sum \eqref{eq:7.1}
are of the form $\bt_{ij}x_ix_j.$ We distinguish two subcases.
\begin{enumerate} \dispace
\item[B)] Exactly two of the terms $\bt_{ij}x_ix_j$ are
$\nu$-dominant. \item[C)] All three such terms are $\nu$-dominant.
\end{enumerate}

In Case B there is a permutation $(i,j,k)$ of $(1,2,3)$ such that
\begin{equation}\label{eq:7.2}
q(x)\cong_\nu\bt_{ij}x_ix_j\cong_\nu\bt_{ik}x_ix_k>_\nu\bt_{jk}x_jx_k,
\end{equation}
while in Case C we have
\begin{equation}\label{eq:7.3}
q(x)\cong_\nu\bt_{12}x_1x_2\cong_\nu\bt_{13}x_1x_3 \cong_\nu\bt_{23}x_2x_j.
\end{equation}

In both cases $q(x)>_\gamma \al_ix_i^2$ for all $i\in I.$ It
follows by \corref{cor:II.9.8} that in Case B both vectors $x(i,j)$
and $x(i,k)$ are $g$-anisotropic and $q$-minimal, while in Case C
all three vectors $x(1,2),$ $x(1,3)$, $x(2,3)$ have these
properties. Due to our knowledge of all $\nu$-dominant terms in
the sum \eqref{eq:7.1}, we see that in Case B
\begin{align*}
b(x(j),x(k))<_\nu q(x(i,j))\cong_\nu q(x(i,k))
&\cong_\nu q(x), \end{align*} while in Case C for every 2-element
subset $\{r,s\}$ of $I$ we have $b(x(r),x(s))\in \tT $ and
$$b (x(r),x(s))\cong_\nu q(x(r,s))\cong_\nu q(x).$$
\{Observe that $b(\veps_i,\veps_i)\le_\nu\al_i$,
cf. \cite[Ineq. (1.9)]{QF1}.\}

\begin{enumerate}
\item[D)] We turn to the case $n=4,$ which is easier. Assume that
$x$ is $q$-minimal. Then we have exactly two $\nu$-dominant terms
in the sum \eqref{eq:7.1}, $\bt_{ij}x_ix_j,$
$\bt_{i\ell}x_kx_\ell,$ with $\{i,j\}$ disjoint from
$\{k,\ell\}$, since otherwise there would exist a set
$S\varsubsetneqq I$ with $q(x(S))=q(x).$ Moreover, these terms are
tangible.

Arguing as above we conclude easily that there is a partition
$I=J\dot\cup K$ with $|J|=|K|=2,$ such that $x(J)$ and $x(K)$ are
$g$-anisotropic and $q$-minimal with
$$q(x(J))\cong_\nu q(x(K))\cong_\nu q(x),$$
while $q(x(S))<_\nu q(x)$ for all other subsets $S$ of $I$ with
$|S|\le 2.$ Also for any two different subsets $S,T$ of $I$ with
$|S|\le 2, $ $|T|\le 2$, including $S=J,$ $T=K,$ we have
$$b(x(S),x(T))<_\nu q(x).$$
\end{enumerate}

Summarizing the essentials of this study, we obtain

\begin{thm}\label{thm:7.2}
Assume that $x$ is $q$-minimal and $\supp(x)=I=\{1,\dots, n\}$
with $n\ge3.$ Then~$x$ is $g$-isotropic and exactly one of the
following four cases holds:

\begin{enumerate} \dispace
\item[A)] $n=3.$ There is a unique partition $I=J\dot\cup K$ with
$|J|=1,$ $|K|=2,$ both $x(J),$ $x(K)$ $g$-anisotropic and
$q$-minimal, and $q(x(J))\cong_\nu q(x(K))\cong_\nu q(x).$

\item[B)] $n=3.$ There are exactly two $2$-element subsets $J$ and
$K$ of $I$ with $x(J),$ $x(K)$ $g$-anisotropic and $q$-minimal and
$q(x(J))\cong_\nu q(x(K))\cong_\nu q(x).$

\item[C)] $n=3.$ For any $2$-element subset $J$ of $I$, the vector
$x(J)$ is $q$-minimal and $g$-anisotropic and $q(x(J))\cong_\nu
q(x).$ Thus the properties listed in B) hold for \textit{any} two
$2$-element subsets $J,K$ of $I.$

\item[D)] $n=4.$ There are exactly two $2$-element subsets $J$ and
$K$ of $I$ such that $x(J),$ $x(K)$ are $g$-anisotropic,
$q$-minimal and
$$q(x(J))\cong_\nu q(x(K))\cong_\nu q(x).$$
$J$ and $K$ are disjoint.
\end{enumerate}

In all four cases, we have $I=J\cup K,$ whence $x=x(J)\vee x(K)$
for the sets $J,K$ from above. Moreover, in Cases A and $D$,
\begin{equation}\label{eq:7.4}
b(x(J),x(K))<_\nu q(x). \end{equation} In Case B,
\begin{equation}\label{eq:7.5}
b(x(J),x(K))= q(x), \end{equation} whereas
\begin{equation}\label{eq:7.6}
b(x(J\sm K),x(K\sm J))\cong_\nu q(x). \end{equation}
In Case C, \eqref{eq:7.5} holds for any two different $2$-element
subsets $J,K$ of $I,$ and moreover
\begin{equation}\label{eq:7.7}
b(x(J\sm K),x(K\sm J))\cong_\nu q(x),\quad
b(x(J\sm K),x(K\sm J)\in \tT.
\end{equation}
\end{thm}
\pSkip
As before we assume that $V$ is free with base
$(\veps_i \ds |i\in I),$ $I=\{1,\dots,n\}$, $n=3$ or~4. Given two
$g$-anisotropic $q$-minimal vectors $y,z\in  V$ of small support,
we now ask for conditions under which the vector $x:=y\vee z$ is
$q$-minimal and has full support $I.$ In view of \thmref{thm:7.2},
we will be content to assume from the beginning that
\begin{equation}\label{eq:7.6.b}
b(y,z)\le_\nu q(y)\cong_\nu q(z).
\end{equation}

A satisfactory converse to \thmref{thm:7.2} in the cases A) and B)
runs as follows.

\begin{thm}\label{thm:7.3}
Assume that $y,z\in V$ are $g$-anisotropic and $q$-minimal,
and furthermore  that $y\vee z$ has full support $I$, and
\begin{equation}\label{eq:7.7.b}
b(y,z)<_\nu q(y)\cong_\nu q(z).
\end{equation}
Assume finally that $n=3,$ $|\supp (y)|=1,$ $|\supp (z)|=2,$ or $n=4,$ and
$|\supp (y)|=|\supp (z)|=2.$ Then $x:=y\vee z$ is $q$-minimal.
\end{thm}

\begin{proof}
We have $\supp (y)\cup\supp (z)=I,$ which forces $\supp(y)\cap\supp
(z)=\emptyset.$
\pSkip

a) Assume first that $n=3.$ After a permutation of the
$\veps_i,$ we may assume
$$y=y_1\veps_1,\qquad z=z_2\veps_2+z_3\veps_3,$$
and then have $x=\sum\limits_1^3x_i\veps_i$ with
$$x_1=y_1,\qquad x_2=z_2,\qquad x_3=z_3.$$
It follows from \propref{prop:II.9.6}.a  and
\corref{cor:II.9.8} that $\al_1x_1^2=q(y)\in\tT$ and
\begin{equation}\label{eq:7.8}
\al_2x_2^2+\al_3x_3^2<_\nu\bt_{23}x_2x_3=q(z)\in\tT.\end{equation} Thus $x_1,x_2,x_3,\al_1,\bt_{23}$ are all
tangible. Further by assumption \eqref{eq:7.7.b}
\begin{equation}\label{eq:7.9}
\bt_{11}x_1^2+\bt_{12}x_1x_2+\bt_{13}x_1x_3<_\nu\al_1x_1^2\cong_\nu\bt_{23}x_2x_3.
\end{equation}
Here $\bt_{11}=b(\veps_1,\veps_1)\le_\nu\al$ (cf.
\cite[Ineq. (1.9)]{QF1}). It follows that
$$q(x)=\al_1x_1^2+\bt_{23}x_2x_3=eq(y)=eq(z).$$
Given $x'=\sum\limits_1^3x_i'\veps_i< x,$ we want to prove
that $q(x')<q(x).$ It suffices to consider the case $x_1'<x_1,$
$x_2'=x_2,$ $x_3'=x_3$ and $x_1'=x_1,$ $x_2'<x_2,$ $x_3'=x_3.$
Notice that $x_i'<x_1$ implies $x_i'<_\nu x_i$ since all $x_i$ are
tangible.

In the first case $\bt_{23}x_2'x_3'=\bt_{23}x_2x_3,$ and we
learn from \eqref{eq:7.8} and \eqref{eq:7.9} that in the sum
$$\sum_1^3\al_ix_i'{}^{2}+\sum_{i<j}\bt_{ij}x_i'x_j'=q(x')$$
there is only one $\nu$-dominant term $\bt_{23}x_2x_3,$ which is
tangible. Thus
$$q(x')=\bt_{23}x_2x_3\in \tT,\quad\text{and}\quad
q(x')\cong_\nu q(x).$$ Since $q(x)$ is ghost, this implies
$q(x') <  q(x).$ In the second case where $x_2'<x_2,$ we can
argue in the same way, now obtaining
$q(x')=\al_1x_2^2\in\tT$ and then $q(x')<q(x).$ Thus $x$
is indeed $q$-minimal.
\pSkip

b) Now let $n=4.$ We may assume that $\supp (y)=\{1,2\}$  and $\supp
(z)=\{3,4\},$ whence
$$y=y_1\veps_1+y_2\veps_2,\qquad
z=z_3\veps_3+z_4\veps_4,$$ and
$x=\sum\limits_1^4x_i\veps_i$ with
$$x_1=y_1,\qquad x_2=y_1,\qquad x_3=z_3,\qquad x_4=z_4.$$
Trivially $y=x(1,2),$ $z=x(3,4).$ We infer from \corref{cor:II.9.8} 
that \begin{gather}
\al_1x_1^2+\al_2x_2^2<_\nu\bt_{12}x_1x_2=q(y)\in \tT,\\
\al_3x_3^2+\al_4x_4^2<_\nu\bt_{34}x_3x_4=q(z)\in \tT,
\end{gather}
and further from Condition \eqref{eq:7.7} that
$$\bt_{13}x_1x_3+\bt_{14}x_1x_4+\bt_{23}x_2x_3+\bt_{24}x_2x_4<_\nu
q(y)\cong_\nu q(z).$$ Let $x'<x,$ and assume w.l.o.g. that exactly
one coordinate $x_i'<x_i,$ say $x_1'<x_1,$ which implies
$x_1'<_\nu x_1.$ If $q(x')=q(x)$ would hold, then
\begin{align*}
q(x')&=\bt_{12}x_1'x_2+\bt_{34}x_3x_4 = \bt_{34}x_3x_4.
\end{align*}
But $q(x')$ is tangible, while $q(x)$ is ghost. This contraction
proves that $q(x')<q(x)$, and we conclude that $x$ is $q$-minimal.
\end{proof}

If $n=3$ and $|\supp (y)|=|\supp (z)|=z,$ then a crude converse to
\thmref{thm:7.2}, analogous to \thmref{thm:7.3} with only
condition \eqref{eq:7.7.b} replaced by \eqref{eq:7.6.b}, does not
hold, as the following example shows.

\begin{examp}\label{examp:7.4}
Let $y=y_1\veps_1+y_2\veps_2$ and
$z=z_1\veps_1 +z_3\veps_3$ with
$y_1,y_2,z_1,z_3\in\tT$ and $ey_1=ez_1,$ $ey_2=ez_3,$ but
$y_1\ne z_1.$ Then
$$x:=y\vee z=x_1\veps_1+x_2\veps_2+x_3\veps_3$$
with
$$x_1=ey_1,\qquad x_2=y_2,\qquad x_3=z_3.$$
Assume further that \begin{enumerate} \dispace

\item[1)]
$\bt_{12},\bt_{13}\in\tT,$

\item[2)]
$\al_1y_1^2+\al_2y_2^2<_\nu\bt_{12}y_1y_2\in\tT,$

\item[3)] $\al_1z_1^2+\al_3z_3^2<_\nu\bt_{13}z_1 z_3.$
\end{enumerate}
Both $y$ and $z$ are $q$-minimal and $g$-anisotropic by
\corref{cor:II.9.8}
, and
$$q(y)=\bt_{12}y_1y_2\cong_\nu\bt_{13}z_1z_3=q(z).$$
Since $\bt_{11}:=b(\veps_1,\veps_1)\le_\nu\al_1$
and $ey_1=ez_1,$ we have
$$b(y_1\veps_1,z_1\veps_1)\le_\nu
\al_1y_1^2\cong_\nu\al_1z_1^2$$ and conclude that
$$b(y,z)=\bt_{11}y_1z_1+\bt_{12}z_1y_2+\bt_{13}y_1z_3=eq(y)=eq(z).$$
Thus Condition \eqref{eq:7.6.b} is valid. We have $x=y+z,$ whence
$$q(x)=q(y)+q(z)+b(y,z)=eq(y).$$
Let now $x':=y_1\veps_1+y_2\veps_2+z_3\veps_3.$
Then $x'<x,$ but
$$q(x')\ge \bt_{12}y_1y_2+\bt_{13}y_1z_3=eq(y).$$
Thus $q(x')=q(x).$ This proves that $x$ is \textit{not}
$q$-minimal.
\end{examp}

The vector $x=y\vee z$ in \thmref{thm:7.3} obviously satisfies
$y=x(J), $ $z=x(K)$ with $J:=\supp (y),$ $K:=\supp (z),$ while for the
vector $y\vee z$ in Example \ref{examp:7.4} this does not hold. If
we insist on the property $y=x(J),$ $z=x(K),$ then we obtain a
converse of \thmref{thm:7.2} also in the cases B) and D) as
follows.

\begin{thm}\label{thm:7.5}
Let $n=3.$ Assume that $y,z\in V$ are $g$-anisotropic and
$q$-minimal with respective support $J,K$ such that $|J| = 2,$ $|K|=2,$ $J\cup
K=I,$ whence $J\cap K$ is a singleton. Assume that $y(J\cap
K)=z(J\cap K)$ and furthermore  that either
\begin{equation}\label{eq:7.10}
b(y(J\setminus K),z(K\setminus J))<_\nu q(y)\cong_\nu q(z);
\end{equation}
or
\begin{equation}\label{eq:7.11}
b(y(J\setminus K),z(K\setminus J))\in \tT,\quad
b(y(J\setminus K),z(K\setminus J))\cong_\nu q(y)\cong_\nu q(z).
\end{equation}
Then $x:=y\vee z$ is $q$-minimal and, of course, $x(J)=y,$
$x(K)=z.$
\end{thm}

\begin{proof} We may assume that $J=\{1,2\},$ $K=\{1,3\},$
 and then
have
$$y=y_1\veps_1+y_2\veps_2,\qquad
z=z_1\veps_1+z_3\veps_3$$ with $y_1=z_1.$ Then
$x=\sum\limits_1^3 x_i\veps_i$ with
$$x_1=y_1 = z_1,\qquad x_2=y_2,\qquad x_3=z_3.$$
It follows from \corref{cor:II.9.8}
 that
\begin{enumerate} \dispace
\item[(1)]
$\al_1x_1^2+\al_2x_2^2<_\nu\bt_{12}x_1x_2=q(y)\in\tT,$

\item[(2)] $\al_1x_1^2+\al_3x_3^2<_\nu
\bt_{13}x_1x_3=q(z)\in\tT.$
\end{enumerate}
Assume that $x'=\sum\limits_1^3x_i'\veps_i$ is given with
either
\begin{align*}
&x_1'<x_1,\quad x_2'=x_2,\quad x_3'=x_3\quad\text{or}\\
&x_1'=x_1,\quad x_2'<x_2,\quad x_3'=x_3.
\end{align*}
We will prove that $q(x')<q(x),$ and then will be done.

Taking into account that
$$b(y(J\setminus K),z(K\setminus
J))=b(y_2\veps_2,z_3\veps_3)=\bt_{23}x_2x_3,$$ we
see that
\begin{enumerate}
\item[(3)] $\bt_{23}x_2x_3<_\nu \bt_{12}x_1x_2\cong_\nu
\bt_{13}x_1x_3,$
\end{enumerate}
while \eqref{eq:7.11} says that
\begin{enumerate}
\item[(4)] $\bt_{23}x_2x_3\in\tT,\quad
\bt_{23}x_2x_3\cong_\nu
\bt_{12}x_1x_2\cong_\nu\bt_{13}x_1x_3.$
\end{enumerate}

Assume that (3) holds. If $x_1'<x_1,$ then $x_1'<_\nu x_1,$ and
thus
$$\bt_{12}x_1'x_2<_\nu \bt_{12}x_1x_2,\quad
\bt_{13}x_1'x_3<_\nu\bt_{13}x_1x_3.$$ It follows from (1),
(2), (3) that $q(x')<_\nu q(x),$ whence $q(x')<q(x).$ If
$x_2'<x_2,$ then $x_2'<_\nu x_2,$ and thus
$$\bt_{12}x_1x_2<_\nu \bt_{12}x_1x_2,\qquad
\bt_{23}x_2'x_3<_\nu\bt_{23}x_2x_3.$$ Now we conclude from
(1), (2), (3) that
$$q(x')=\bt_{13}x_1x_3\cong_\nu q(x).$$
But $q(x')\in\tT,$ $q(x)\in\tG,$ and so
$q(x')<q(x)$ again.

Assume finally that (4) holds. If $x_1'<x_1$, we see by the same
reasoning that $$q(x')=\bt_{23}x_2x_3\cong_\nu q(x),$$ while if
$x_2'<x_2$ then
$$q(x')=\bt_{13}x_1x_3\cong_\nu q(x).$$
In both cases $q(x')\in\tT,$ $q(x)\in\tG,$ and so
$q(x')<q(x).$ This completes the proof that $x$ is $q$-minimal.
\end{proof}

We complement Theorems \ref{thm:7.2}, \ref{thm:7.3}, \ref{thm:7.5} by an observation on certain
pairs of $q$-minimal vectors.

\begin{thm}\label{thm:7.6}
Assume that $x,y\in V$ are $q$-minimal vectors with $y<x$ and $q(y)\cong_\nu q(x).$ Let $ J :=\supp (y).$ Then $q(y)\in \tT,$ $q(x)\in\tG,$ and one of the following cases holds:
\begin{enumerate}\dispace
\item[1)] $|\supp (y)|=|\supp (x)|=1,$ $x=ey.$
\item[2)] $|\supp (y)|=|\supp (x)|=2,$ $y<x<ey.$
\item[3)] $|\supp (y)|=1,$ $|\supp (x)|\ge 2,$ $y=x( J ).$
\item[4)] $|\supp (y)|=2,$ $|\supp (x)|\ge 3,$ $y=x( J ).$
\end{enumerate}
\end{thm}

\begin{proof} a) We may assume that $\supp (x)=\{1,\dots,n\}.$ We have $q(y)<q(x)$ because $x$ is $q$-minimal. This forces $q(y)\in \tT,$ $q(x)\in\tG.$ \pSkip

b) Assume $n=1.$ Now $y=y_1\veps_1,$ $x=x_1\veps_1,$ and $\al_1^2y_1\in \tT ,$ $e\al_1^2y_1=\al_1^2y_1\in\tT,$ $e\al_1^2y_1=\al_1^2x_1.$ This implies $x_1=ey_1,$ whence $x=ey.$
\pSkip

c) Suppose that $| J |=1,$ $n\ge 2.$ We may assume that $J=\{1\}.$ Now $y=y_1\veps_1,$ $\al_1y_1^2\in\tT$ and $y_1\le x_1,$ whence $\al_1y_1^2\le \al_1x_1^2.$ Since $q(y)\cong_\nu q(x),$ the terms $\al_1x_1^2$ is $\nu$-dominant in the sum
\begin{equation}\label{eq:7.12}
\sum_1^n\al_ix_i^2+\sum_{i<j}\bt_{ij}x_ix_j=q(x)
\end{equation}
Since $x$ is $q$-minimal, this forces $\al_1x_1^2\in\tT$ and then $\al_1y_1^2=\al_1x_1^2.$ We conclude that $y_1=x_1,$ i.e., $y=x(1).$
\pSkip
d) Suppose that $| J |=2,$ $n\ge 2.$ We may assume that $ J =\{1,2\}.$ By \corref{cor:II.9.8}, $$\al_1y_1^2+\al_2y_2^2<\bt_{12}y_1y_2 = q(y)\in \tT.$$
It follows from $q(y)\cong_\nu q(x)$ and $y_1\le x_1,$ $y_2\le x_2 $ that $\bt_{12}x_1x_2$ is a $\nu$-dominant term in the sum \eqref{eq:7.12} and $\bt_{12}x_1x_2\cong_\nu\bt_{12}y_1 y_{2},$ $\bt_{12}x_1x_2\ge\bt_{12}y_1y_2.$

If $n>2,$ then the $q$-minimality of $x$ forces $\bt_{12}x_1x_2\in \tT,$ and we conclude from $y_1\le x_1$, $y_2\le x_2$ that $y_1=x_1,$ $y_2=x_2,$ i.e., $y=x(1,2).$

If $n=2,$ we conclude from $q(y)<q(x)$ that $e \bt_{12}y_1y_2=\bt_{12}x_1x_2,$ and then that $y_1\cong_\nu x_1,$ $y_2\cong_\nu x_2,$ whence $ex=ey.$
But $x\ne ey,$ since the vector $ey$ is not $q$-minimal. Thus either $x_1=ey_1,$ $x_2=y_2,$ or $x_1=y_1,$ $x_2=ey_2.$ We conclude that $y<x<ey.$
\end{proof}


\begin{thebibliography}{10} 




%

\bibitem{IKR1}
 Z.~Izhakian, M.~Knebusch, and L.~Rowen.
\newblock Supertropical semirings and supervaluations.
 \newblock  \emph{J. Pure and Appl.~Alg.}, 215(10):2431--2463, 2011.


\bibitem{IzhakianKnebuschRowen2010LinearAlg}
Z.~Izhakian, M.~Knebusch, and L.~Rowen.
\newblock Supertropical linear algebra.
\newblock  {\em Pacific J. of Math.},   266(1):43-–75, 2013.




\bibitem{IKR-LinAlg2}  Z.~Izhakian, M.~Knebusch, and   L.~Rowen.
\newblock Dual spaces and bilinear forms in supertropical linear
algebra, {\em Linear and Multilinear Algebra}, 41(7):2736--2782, 2013.

\bibitem{QF1}
Z.~Izhakian, M.~Knebusch, and L.~Rowen.
\newblock Supertropical quadratic froms I. {\em Journal of Pure and Applied Algebra}, to appear.
\newblock (Preprint at arXiv:1309.5729.v2, 2015.)


\bibitem{QF3}
Z.~Izhakian, M.~Knebusch, and L.~Rowen.
\newblock Supertropical quadratic forms III, in preparation.



\bibitem{IzhakianRowen2007SuperTropical}
Z.~Izhakian and L.~Rowen.
\newblock {Supertropical algebra}.
\newblock   {\em Adv. in Math.}, 225(4):2222--2286, 2010.


%
%
%
%



\bibitem{Spez} M. Knebusch. \textit{Specialization of Quadratic and Symmetric Bilinear Forms}, Springer London, 2010.



 \bibitem{lam} T. Y. Lam.
\textit{Orderings, Valuations and Quadratic Forms}, Regional Conference Series in Applied Mathematics 52, Amer. Math. Soc.,  1981.

 \end{thebibliography}
\end{document}